\documentclass[12pt]{article}

\usepackage[top=1in,bottom=1in,left=1in,right=1in]{geometry}
\usepackage{color}
\usepackage{caption,subfigure}
\usepackage{graphicx}
\usepackage{upgreek}
\usepackage{bbm}
\usepackage [english]{babel}
\usepackage [autostyle, english = american]{csquotes}
\MakeOuterQuote{"}
\usepackage[center]{titlesec}
\titleformat{\section}[hang]{\normalfont\scshape}{\thesection.}{.5em}{\filcenter}[]
\titleformat{\subsection}[hang]{\normalfont\scshape}{\thesubsection.}{.5em}{\filcenter}[]
\titleformat{\subsubsection}[hang]{\normalfont\scshape}{\thesubsubsection.}{.5em}{\filcenter}[]

\usepackage{amsmath,mathrsfs}
\usepackage{amssymb,amsthm}
\usepackage[all]{xy}

\usepackage{graphicx}
\usepackage{paralist}
\usepackage{comment}
\usepackage{tikz}
\usetikzlibrary{positioning,arrows.meta}
\usetikzlibrary{positioning, shapes}

\usepackage{hyperref}
\numberwithin{equation}{section}
\newtheorem{thm}{Theorem}[section]
\newtheorem{corollary}{Corollary}[section]
\newtheorem{prop}{Proposition}[section]

\theoremstyle{definition}

\newtheorem{remark}{Remark}[section]

\titleformat{\subsubsection}[runin]{\normalfont\bfseries}{\thesubsubsection}{1em}{}

\newcommand\restr[2]{{% we make the whole thing an ordinary symbol
  \left.\kern-\nulldelimiterspace % automatically resize the bar with \right
  #1 % the function
  \vphantom{\big|} % pretend it's a little taller at normal size
  \right|_{#2} % this is the delimiter
  }}

\def\A{\mathcal{A}}

\begin{document}
\baselineskip 1.5em

\title{Well-posedness and Bilinear Controllability  of a Repairable  System with Degraded State}
% \author{ Daniel Owusu Adu and Weiwei Hu }

  \author{Daniel Owusu Adu \thanks{Department of Mathematics, University of Georgia, Athens, GA 30602, USA  (e-mail: daniel.adu@uga.edu) }
  \and  Weiwei Hu  \thanks{Department of Mathematics, University of Georgia, Athens, GA 30602, USA  (e-mail: Weiwei.Hu@uga.edu)}
    }

\maketitle

\begin{abstract}
In this work, we consider the dynamics of  repairable systems characterized by three distinct states: one signifying normal operational states, another representing degraded conditions and a third denoting failed conditions. These systems are characterized by their ability to  be repaired  when failures and/or degradation occur. Typically described by transport equations,  these systems exhibit a coupled nature, interlinked through integro-differential equations and integral boundary conditions that dictate the transitions among all the states. In this paper,  we address two less-explored facets: 1) the well-posedness and the asymptotic behavior of such systems with maximum repair time being finite; and 2) the bilinear controllability of the system via repair actions. In particular, we focus on the case where only one degraded and one failed states exist. We first discuss  part 1) for given time-independent  repair rates and then  design  the space-time dependent repair strategies that can manipulate system dynamics to achieve the desired level over a finite horizon. Our objective is to enhance the system availability- the probability of being operational when needed over a fixed period of time. We present rigorous analysis and develop control strategies that leverage the bilinear structure of the system model.
\end{abstract}

%, demonstrating that appropriate control actions can manipulate system dynamics to achieve desired availability levels over a finite horizon.

%\begin{keyword}
%Reparable system,  nonreflexive Banach space, bilinear control,

%\end{keyword}

%\end{frontmatter}



\section{Introduction}
In real life practice, systems often encounter challenges stemming from failures or degradation. These issues are prevalent across various applications such as product design, inventory systems, computer networking, electrical power systems, and complex manufacturing processes. All these systems are susceptible to degradation or failure but are able to be  restored to satisfactory operation through repair actions (e.g.,~\cite{Sandler, jardine2013maintenance, moubray1997reliability, AJ-AT:13, JM:97}). Repairable systems are capable of  undergoing repair/maintenance actions
 when failures and/or degradation occur.   Our current work  focuses on a  three-state repairable system, characterized by its  transitions among three states:  the functioning (good), the degraded, and the failed ones.  The digram of the model is presented in Fig.\ref{transition}.
 \begin{figure}
\centering
\resizebox{0.5\textwidth}{!}{
        \begin{tikzpicture}[node distance=4cm, auto]

% Define nodes for states
    \node[circle, draw=black, fill=white] (state1) {0};
    \node[circle, draw=black, fill=white, right of=state1] (state2) {1};
    \node[circle, draw=black, fill=white, right of=state2] (state3) {2};

    % Connect the states
    \draw[->] (state1) -- node[midway, above] {$\lambda_1$} (state2);
    \draw[->] (state2) -- node[midway, above] {$\lambda_{2}$} (state3);
    \draw[->, bend left] (state1) to node[midway, above] {$\lambda_{2}$} (state3);
    \draw[->, bend left] (state2) to node[midway, above] {$\mu_1$} (state1);
    \draw[->, bend left] (state3) to node[midway, below] {$\mu_2$} (state1);

  \end{tikzpicture}
}
\caption{Transition digram of a three-state repairable system. Here $0,1,2$ denote good, degraded and failed states. Good state can degrade or fail with rates $\lambda_1$ and $\lambda_{2}$, respectively. Degraded and failed states can be repaired at rates $\mu_1$ and $\mu_2$, respectively. The degraded state can also fail with rate 
$\lambda_{2}$.}
\label{transition}
\end{figure}
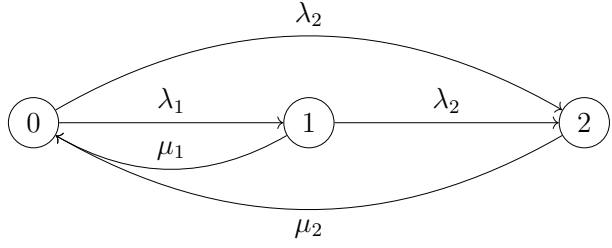

The mathematical model considered here is  described by transport equations, collectively interwoven through an integro-differential equation. Moreover, an integral boundary condition is prescribed to the transport equation which  is a crucial element that governs the transitions of the states.
 This type of mathematical models for repairable systems was derived using Markov chain and supplementary variable techniques (e.g.,~\cite{chung1981reparable, 
 cox1955analysis,  gupta1984cost, gupur2011functional}).   Specifically, we consider the 
 mathematical model proposed by  Gupta and Agarwal in \cite{gupta1984cost}
 \begin{eqnarray}
\left\{\begin{array}{l}
\displaystyle\dot{p}_0(t) =-\sum^{2}_{i=1}\lambda _i  p_0(t)+\sum^2_{i=1} \int_0^{L} \mu_i(x)p_i(x,t) dx  \\
\displaystyle\frac{\partial{p_1(x,t)}}{\partial{t}} + \frac{\partial{p_1(x,t)}}{\partial{x}} =-(\mu_1(x)+\lambda_{2})p_{1}(x,t)\\
\displaystyle\frac{\partial{p_2(x,t)}}{\partial{t}} + \frac{\partial{p_2(x,t)}}{\partial{x}} =-\mu_2(x)p_{2}(x,t)
\end{array}\right. \label{system}
\end{eqnarray}
%\begin{align}
%&\frac{d p_{0}(t)}{d t}=-(\lambda_{1}+\lambda_{2}) p_{0}(t) + \int^{L}_{0} \mu_{1}(x)p_{1}(x,t) \,d x+ \int^{L}_{0} \mu_2(x)p_{2}(x,t) \,d x,    \\
%&\frac{\partial{p_{1}(x,t)}}{\partial{t}} + \frac{\partial{p_{1}(x,t)}}{\partial{x}} =-(\mu_{1}(x)+\lambda_{2})p_1(x,t) , \\
%&\frac{\partial{p_2(x,t)}}{\partial{t}} + \frac{\partial{p_2(x,t)}}{\partial{x}} =-\mu_2(x)p_2(x,t), 
%\end{align}
with  boundary conditions 
\begin{align}
p_{1}(0,t)&=\lambda_{1} p_0(t), \label{sys_BC_D}\\
 p_2(0,t)&=\lambda_{2} p_0(t)+\lambda_{2}\int^{L}_{0} p_{1}(x, t)\, dx, 
\label{sys_BC_F}
\end{align}
and initial conditions
\begin{equation}
p_0(0)=\phi_0\geq0, \quad p_{1}(x, 0)=\phi_1(x)\geq0, \quad p_2(x, 0)=\phi_2(x)\geq0.
\label{sys_IC}
\end{equation} 
%where  $\phi_G, \phi_D(x), \phi_F(x)\geq0$ for $x\in(0, L)$, and 
%\begin{align}
%\phi_G+\int^L_0\phi_D(x)d x+\int^L_0\phi_F(x)d x=1.\label{prop_ini}
%\end{align}
Here 
\begin{enumerate}[1)]
\item 
$\lambda_{1}>0$ and  $\lambda_{2}>0$ represent the failure  rates of the system from the good mode to  the degraded and  to  the  failed states, respectively  (see Fig.~\ref{transition}). It is assume that the system has the same failure rate  $\lambda_{2}$ from the degraded  to   the  failed state. 
%\item  $\lambda_{1}\geq0$ represents the failure  rate of the system from the degraded mode to   the  failed state;
%\item  $x\in[0, L]$ represents the  time since repair begins, where $L>0$ is the maximum repair time; 
\item $\mu_{1}(x)\geq 0$ and  $\mu_2(x)\geq 0$ represent the
 repair  rates of system in  the degraded and the failed states with an elapsed repair time $x\in[0, L]$ for $0<L<\infty$, respectively. Assume that 
 \begin{align}
&\int^l_0\mu_{i}(x)\,dx<\infty, \quad\text{ for}\quad  0<l<L, \label{mu1}\\
\text{and} \quad &\int^L_0\mu_{i}(x)\,dx=\infty, \quad i=1,2. \label{mu2}
\end{align}
We further assume that there exists some positive integer  $N\in\mathbb{N}^+$ such that 
 \begin{align}
\lim_{x\to L}\mu_i(x)(L-x)^N<\infty, \quad i=1,2. \label{mu3}
\end{align}
 \item  $p_0(t)$ represents the probability of the system in good state at time $t$;
 
 \item  $p_{1}(x, t)$ and $p_2(x, t)$ represent the probability density distributions  of the system in the degraded and  the failed  states, respectively, at time $t$  with an elapsed repair time $x$; 
 %\item $p_2(x, t)$ represents the probability density distribution  of the system in the failed state  at time $t$, which is under repair with an elapsed repair time  $x$.
\item Let $\hat{p}_{1}(t)$  and $\hat{p}_0(t)$ denote the probabilities of the system in the degraded and the failed states at time $t$, respectively. We have
   \begin{align}
            \hat{p}_{1}(t)=\int^{L}_0p_{1}(x,t)\, d x\quad \text{and} \quad    \hat{p}_2(t)=\int^{L}_0p_2(x,t)\, d x. \label{prob_p0}
    \end{align}
\item
 The initial probability distributions of the system  satisfy 
\begin{align}
 \phi_{0}+\sum^2_{i=1}\int^{L}_0\phi_{i}(x)\, dx=1. \label{SUM_ini}
 \end{align}
 \end{enumerate}
  {\it Assumptions  associated with the system.} (1) The failure and degradation rates are constant.
(2) All failures are statistically independent.
(3) The system can fail in the degraded state and
the failure rate is the same regardless of whether the
system is good or degraded.
(4) The repair time for the degraded or the failed device  is arbitrarily distributed.
(5) The repair process begins soon after the device is in failure state.
(6) The system has only one repair facility and
repair is to the same quality as new. Repair never
damages anything.
%(6) The repaired device is as good as new.
 
%(7) No further failure can occur when the device has been down.

In \cite{gupta1984cost},  Laplace transport was used to solve system \eqref{system}--\eqref{sys_IC} and its steady-state  without discussing  the well-posedness and the stability of the system model.  Recently in~\cite{WH-TA-XB-ZQ: 23}, Hu~{\it et\,al.} discussed the degradation and failure rates identification of this system, however, the well-posedness issues were not in their scope. 
In the current work, we will first rigorously address these issues  using $C_0$-semigroup theory with  repair rates being given and time-independent. Since repair actions play an essential role in ensuring the system performance, it is natural to employ them as control inputs for the system. Furthermore, due to the bilinear structure between  the repair rates and the system states. This immediately  leads to a bilinear control problem.  While the optimal bilinear open-loop control design of similar  systems  has been considered in \cite{boardman2019optimal, WH:22}, the  bilinear controllability keeps open.   Our  main objective of this  work  is  to establish  this result  of the system via repair actions, where we assume that 
repair actions are allowed to be time-dependent  and we propose explicit feedback laws for the actions  to steer the system behavior to the desired  one at a given final time. 
 
The rest of this paper is organized as follows. In Section \ref{sec_wellposedness}, we first establish  the well-posedness of system~\eqref{system}--\eqref{sys_IC}
by showing that the system operator generates a  positive $C_0$-semigroup of contraction. Moreover, it can be shown that zero is a simple eigenvalue of the generator and the only spectrum on the imaginary axis. Lastly, by showing the eventual compactness of the semigroup, we  are able to obtain the exponential  convergence of the time-dependent solution of the system to its steady-state. In Section \ref{sec_control}, we employ  the repair actions as the system control inputs and  construct explicit space-time dependent repair rates in  feedback forms. The closed-loop system shares similar attributes of the open-loop. Finally, we establish the bilinear controllability of the system by making use of its exponential stability and conclude our work  in Section \ref{conclusion}.  

\section{Well-posedness and Asymptotic Behavior  of the System}\label{sec_wellposedness}

In this section, we address the well-posedness of system~\eqref{system}--\eqref{sys_IC} and its asymptotic  stability  using $C_0$-semigroup theory.
 
 Let $X = \mathbb{R}\times L^1(0,L)\times L^1(0, L)$ be equipped with the norm
$
\|\cdot\|_{X} := |\cdot | +\sum^2_{i=1}\|\cdot\|_{L^1(0, L)}
$. 
It is clear that the Banach space structure $(X, \|\cdot\|)$ is compatible with the
order structure $(X,\leq)$, that is,  $|f|\leq |g|$ for $f, g \in X$ implies $f\leq g$, thus
 $X$ is a Banach lattice.
Let $ \vec{p}(t)=(p_0(t), p_1(\cdot,t), p_2(\cdot,t))^{\mathrm{T}}$. Then  system equations \eqref{system}--\eqref{sys_IC} can be written
 as an abstract Cauchy problem in $X$:
\begin{align}
\left\{\begin{array}{ll}
\dot{\vec{p}}(t)=\mathcal{A}\vec{p} (t),\\
\vec{p}_0=(\phi_0,\phi_1,\phi_2)^T,
\end{array}
\right. \label{IVP}
\end{align}
where the system operator $\mathcal{A}\colon D(\mathcal{A})\subset X \to X$ is defined as
%\begin{equation}\label{def_A}
%\mathcal{A}\vec{p} = \begin{pmatrix}
%    -\sum^{2}_{i=1}\lambda_i p_0 + \sum_{i=1}^2\int_0^{L} \mu_i(x)p_i(x) \,dx \\
%    -(\frac{d}{dx} +\mu_{1}(x) + \lambda_{2})p_1(x) \\
%    -(\frac{d}{dx} +\mu_2(x))p_{2}(x)
%\end{pmatrix},
%\end{equation}
\begin{align}
\mathcal{A}= \displaystyle 
\begin{pmatrix}
   -\sum^{2}_{i=1}\lambda_i &  \int^L_0\mu_1\cdot \,dx&\int^L_0\mu_2\cdot \,dx \\ 
0  & -(\frac{d}{dx} +\mu_{1}(x) + \lambda_{2}) &0\\
0&0&    -(\frac{d}{dx} +\mu_2(x))
\end{pmatrix},
 \label{def_A}
\end{align}
with domain
\begin{align}\label{D_A}
\mathcal{D}(\mathcal{A}) = \bigg\{\vec{p}=&(p_0, p_1(\cdot), p_2(\cdot))^T\in X\colon p_i\in W^{1,1}(0, L),  \int^L_0 \mu_ip_i\,dx<\infty,\quad i=1,2,\nonumber\\
&\quad\text{and}\ \begin{pmatrix}p_1(0)\  p_2(0)\end{pmatrix}^{T} = \Gamma_1 p_0(x) + \int_0^L\Gamma_2 p_1(x)dx\bigg\},
\end{align} 
where
\begin{align}
\Gamma_1 = \begin{pmatrix}
    \lambda_1 \\
    \lambda_{2}
\end{pmatrix}\quad\text{and}\quad \Gamma_2 = \begin{pmatrix}
    0 \\
     \lambda_{2}
\end{pmatrix}.
\label{oper_BC}
\end{align}

The well-posedness of  the repairable systems governed by the coupled transport and differential-integral equations has been well-studied in literature  using $C_0$-semigroup theory, where the majority focuses on the case  that  the maximum repair time $L=\infty$ (e.g.~\cite{WH-HX-JY-GTZ:07,WH:07,WH:22, gao2022stability,   hu2007exponential, Guo-1,gupur2011functional, Gupur-1,  HG-1, HR-1, HR-2,  xu2005asymptotic, HXYZ-1}). However, in real life applications it  is more realistic to assume $L<\infty$, as no system can be under repair forever. In this case,  the repair rates $\mu_i, i=1,2$ would have a singularity at $x=L$, as shown in the assumptions \eqref{mu1}--\eqref{mu2}. This leads to some different system properties compared to the situation  when $L=\infty$ (e.g.,~\cite{WH:16}). In this work, we consider that $L$ is  finite and  provide a complete proof of the well-posedness of the system \eqref{IVP}   for the convenience of the reader.  

Let $\sigma(\mathcal{A})$ and $\rho(\mathcal{A})$ denote  the spectrum and the resolvent set of $\mathcal{A}$,  respectively.  To start with, we first show that the  right open half-plane is contained in the  resolvent set of $\mathcal{A}$.
 %The proof is a bit long but elementary so we present it in Appendix.

\begin{prop}\label{prop1}
For  the operator $\mathcal{A}$ with its domain $D(\mathcal{A})$ defined in \eqref{def_A}--\eqref{D_A}, we have the following statements hold:
\begin{enumerate}[(1)]
\item $\rho(\mathcal{A})$ contains the set
\begin{equation}\label{eq: spectrum set}
\Psi=\{r\in\mathbb{C} : \mathrm{Re}(r)>0 \quad\text{and}\quad r=ia,\quad a\neq 0\};
\end{equation}
\item the resolvent operator $\mathcal{R}(r, \mathcal{A})$ is compact for $r\in \rho(\mathcal{A})$;
\item  zero is a simple eigenvalue of $\mathcal{A}$ and the only real eigenvalue. 
\end{enumerate}
\end{prop}
\begin{proof}
To analyze the  properties of the resolvent  set  and the  resolvent operator, we first let   $r\in \Psi$,  $\vec{y}(\cdot)=(y_0, y_1(\cdot), y_2(\cdot))^T\in X$, and consider the operator equation 
\[(rI-\mathcal{A})\vec{p}=\vec{y},\]
that is,
\begin{align}
&\Big(r+\sum^{2}_{i=1}\lambda _i\Big)p_0-\sum_{i=1}^2\int_0^{L} \mu_i(x)p_i(x) \,dx=  y_0,\label{eq: solution for good}\\    
&\frac{dp_1(x)}{dx} +(r+\mu_{1}(x)+\lambda_{2})p_1(x)=y_1(x),\label{eq:Degrade}\\
&\frac{dp_2(x)}{dx}+(r+\mu_2(x))p_2(x)=y_{2}(x),\label{eq:Failure}
\end{align}
with the boundary conditions
\begin{align}
p_{1}(0)&=\lambda_{1} p_0, \quad  \quad p_2(0)=\lambda_{2} p_0+\lambda_{2}\int^{L}_{0} p_{1}(x)\, dx.
\label{IBC}
\end{align}
Solving \eqref{eq:Degrade} and  \eqref{eq:Failure} yields
\begin{align}
p_1(x)=&\lambda_1 p_0e^{-\int_0^x(r+\lambda_{2}+\mu_1(s))\,ds}+\int_0^xe^{-\int_{\tau}^x(r+\lambda_{2}+\mu_1(s))\,ds}y_1(\tau)\,d\tau
\label{oper_p1}
\end{align}
and
\begin{align}
p_2(x)
%====notes======
%=&\Big(\lambda_{2} p_0+\lambda_{2}\big(\int^{L}_{0}( \lambda_{1} p_0e^{-\int_0^x(r+ds\mu_1(s)+\lambda_{2})ds}
%+\int_0^xe^{-\int_{\tau}^x(r+ds\mu_1(s)+\lambda_{2})ds}y_1(\tau)\,d\tau)\, dx\big)\Big)\cdot e^{-\int_0^x(r+\mu_2(\alpha))ds}\\
%&+\int_0^xe^{-\int_{\tau}^x(r+\mu_2(\alpha))ds}y_2(\tau)d\tau\\
%====notes======
=&\lambda_{2} p_0e^{-\int_0^x(r+\mu_2(s))\,ds}+\lambda_{2}\lambda_{1} p_0e^{-\int_0^x(r+\mu_2(s))\,ds}\int^{L}_{0} 
e^{-\int_0^x(r+\lambda_{2}+\mu_1(s))\,ds}\,d x\nonumber\\
&+\lambda_{2}e^{-\int_0^x(r+\mu_2(s))\,ds}\int^{L}_{0}\int_0^xe^{-\int_{\tau}^x(r+\lambda_{2}+\mu_1(s))\,ds}y_1(\tau)d\tau\, dx\nonumber\\
&+\int_0^xe^{-\int_{\tau}^x(r+\mu_2(s))\,ds}y_2(\tau)d\tau. \label{oper_p2}
%=&\lambda_{2} p_0e^{-\int_0^x(r+\mu_2(\alpha))\,ds}(1+\lambda_{1} \int^{L}_{0} 
%e^{-\int_0^x(r+ds\mu_1(s)+\lambda_{2})\,ds}\,d x)\nonumber\\
%&+\lambda_{2}e^{-\int_0^x(r+\mu_2(\alpha))ds}\int^{L}_{0}\Psi_1(r)y_1\, dx+\Psi_2(r)y_2.\nonumber
\end{align}
%%=====details======
%To simplify the formulation, we let
%\begin{align*}
%\Phi_1(r)=\int_0^{L} \mu_1(x)e^{-\int^x_0(r+ds\mu_1(s)+\lambda_2)\,ds}\,dx, \quad
%\Phi_2(r)= \int_0^{L} \mu_2(x) e^{-\int_0^x(r+\mu_2(\alpha)\,dx)\,ds}\,dx,
%\end{align*}
%and define
%\begin{align*}
%\Psi_1(r)y_1=\int^x_0e^{-\int_{\tau}^x(r+ds\mu_1(s)+\lambda_{2})\,ds}y_1(\tau)\,d\tau, 
%\quad
%\Psi_2(r)y_2=\int_0^xe^{-\int_{\tau}^x(r+\mu_2(\alpha))ds}y_2(\tau)d\tau,
%\end{align*}
%for $r\in \mathbb{C}$. Note that 
%\begin{align}
%&\int_0^{L} \mu_i(x)e^{-\int^x_0(r+\mu_i(s))\,ds}\,dx
%=-\int_0^{L}e^{-rx}d\,e^{-\int^x_0\mu_i(s)\,ds}\nonumber\\
%&\qquad=-(e^{-rL}e^{-\int^L_0\mu_i(s)\,ds}-1)-r\int^L_0e^{-\int^x_0(r+\mu_i(s))\,ds}\,dx\label{1EST_phi}\\
%&\qquad=1-r\int^L_0e^{-\int^x_0(r+\mu_i(s))\,ds}\,dx, \label{2EST_phi}
%\end{align}
%where from \eqref{1EST_phi} to  \eqref{2EST_phi} we used the fact that $\int^L_0\mu_i(s)\,ds=\infty$.
%Thus
%\begin{align*}
%\Phi_1(r)=1-(r+\lambda_2)\int^L_0e^{-\int^x_0(r+ds\mu_1(s)+\lambda_2)\,ds}\,dx
%\end{align*}
%and
%\begin{align*}
%\Phi_2(r)= 1-r\int_0^{L} e^{-\int_0^x(r+\mu_2(\alpha)\,dx)\,ds}\,dx,
%\end{align*}
%%=====details======
Substituting \eqref{oper_p1}--\eqref{oper_p2} in~\eqref{eq: solution for good} gives 

\begin{equation*}%\label{eq: original solution of good state}
\Phi(r)p_0=F_r(y),    
\end{equation*}
where
%%===notes======
% \begin{align}
%&\left(r+\sum^{2}_{i=1}\lambda _i\right)p_0-\int_0^{L} \mu_1(x)\lambda_1 p_0e^{-\int_0^x(r+ds\mu_1(s)\,dx+\lambda_{2})\,ds}\,dx\\
%&-\int_0^{L}\mu_1(x) \int_0^xe^{-\int_{\tau}^x(r+ds\mu_1(s)+\lambda_{2})\,ds}y_1(\tau)\,d\tau\,dx\\
%&-\int_0^{L} \mu_2(x) \lambda_{2} p_0e^{-\int_0^x(r+\mu_2(\alpha))\,ds}\,dx\\
%&-\lambda_{2}\lambda_{1} p_0  \int^L_0\mu_2(x)e^{-\int_0^x(r+\mu_2(\alpha))\,ds}\,dx
%\int^{L}_{0} e^{-\int_0^x(r+ds\mu_1(s)+\lambda_{2})\,ds}\,d x\\
%&-\lambda_{2}\int_0^{L} \mu_2(x) e^{-\int_0^x(r+\mu_2(\alpha))ds}\,dx\int^{L}_{0}\int_0^xe^{-\int_{\tau}^x(r+ds\mu_1(s)+\lambda_{2})ds}y_1(\tau)d\tau\, dx\\
%&-\int_0^{L}\mu_2(x)\int_0^xe^{-\int_{\tau}^x(r+\mu_2(\alpha))ds}y_2(\tau)d\tau\,dx=  y_0
%\end{align}
%%===notes======
\begin{align*}%\label{eq: spec function}
\Phi(r)=&r+\sum^{2}_{i=1}\lambda _i-\lambda_1 \int_0^{L} \mu_1(x) e^{-\int_0^x(r+\lambda_{2}+\mu_1(s))\,ds}\,dx\\
&- \lambda_{2}\int_0^{L} \mu_2(x)e^{-\int_0^x(r+\mu_2(s))\,ds}\,dx\\
&-\lambda_{2}\lambda_{1}\left( \int^L_0\mu_2(x)e^{-\int_0^x(r+\mu_2(s))\,ds}\,dx\right)
\left(\int^{L}_{0} e^{-\int_0^x(r+\lambda_{2}+\mu_1(s))\,ds}\,d x \right)
%&=r\Big(1+\lambda_1 \int_0^{L} e^{-\int_0^x(r+ds\mu_1(s)\,dx+\lambda_{2})\,ds}\,dx
%+ \lambda_{2}\int_0^{L} e^{-\int_0^x(r+\mu_2(\alpha))\,ds}\,dx\\
%&+\lambda_{2}\lambda_{1}  \int^L_0e^{-\int_0^x(r+\mu_2(\alpha))\,ds}\,dx
%\int^{L}_{0} e^{-\int_0^x(r+ds\mu_1(s)+\lambda_{2})\,ds}\,d x\Big) 
\end{align*}
and
\begin{align}
F_r(y)=&y_0+\int_0^L\mu_1(x)\int_0^xe^{-\int_{\tau}^x(r+\lambda_{2}+\mu_1(s))\,ds}y_1(\tau)d\tau\,dx\nonumber\\
&+\lambda_{2}\left(\int_0^{L} \mu_2(x)e^{-\int_0^x(r+\mu_2(s))\,ds}\,dx\right)\left(\int^{L}_{0}\int_0^xe^{-\int_{\tau}^x(r+\lambda_{2}+\mu_1(s))\,ds}y_1(\tau)d\tau\, dx\right)\nonumber\\
&+\int_0^{L} \mu_2(x)\int_0^xe^{-\int_{\tau}^x(r+\mu_2(s))\,ds}y_2(\tau)\,d\tau\,dx.  \label{def_F}
%&=y_0+\int_0^L\mu_1(x)\Psi_1(r)y_1(x)\,dx+\lambda_{2}\Phi_2(r) \int^{L}_{0}\Psi_1(r)y_1(x)\, dx\\
%&\quad+\int_0^{L} \mu_2(x)\Psi_2(r)y_2(x)\,dx\\
%&=y_0+\int_0^L(\mu_1(x)+\lambda_{2}\Phi_2(r))\Psi_1(r)y_1(x)\,dx\\
%&\quad+\int_0^{L} \mu_2(x)\Psi_2(r)y_2(x)\,dx
\end{align}
Note that  for any $r\in \mathbb{C}$,
\begin{align}
&\int_0^{L} \mu_i(x)e^{-\int^x_0(r+\mu_i(s))\,ds}\,dx
=-\int_0^{L}e^{-rx}d\,e^{-\int^x_0\mu_i(s)\,ds}\nonumber\\
&\qquad=-(e^{-rL}e^{-\int^L_0\mu_i(s)\,ds}-1)-r\int^L_0e^{-\int^x_0(r+\mu_i(s))\,ds}\,dx\label{EST_mu}\\
&\qquad=1-r\int^L_0e^{-\int^x_0(r+\mu_i(s))\,ds}\,dx, \quad i=1,2,\label{1EST_mu}
\end{align}
where from \eqref{EST_mu} to  \eqref{1EST_mu} we used the assumption  that $\int^L_0\mu_i(s)\,ds=\infty, i=1,2$.
In particular,  when $r=0$,
\begin{align}
\int_0^{L} \mu_i(x)e^{-\int^x_0\mu_i(s)\,ds}\,dx=1. \label{2EST_mu}
\end{align}

Thus 
\begin{align}
\Phi(r)=&r+\sum^{2}_{i=1}\lambda _i-\lambda_1\Big(1-(r+\lambda_2) \int_0^{L} e^{-\int_0^x(r+\mu_1(s)\,dx+\lambda_{2})\,ds}\,dx\Big)\nonumber\\
&- \lambda_{2}\Big(1-r\int_0^{L} e^{-\int_0^x(r+\mu_2(s))\,ds}\,dx\Big)\nonumber\\
&-\lambda_{2}\lambda_{1} \Big(1-r \int^L_0e^{-\int_0^x(r+\mu_2(s))\,ds}\,dx\Big)
\int^{L}_{0} e^{-\int_0^x(r+\mu_2(s)+\lambda_{2})\,ds}\,d x \nonumber\\
=&r\Big(1+\lambda_1 \int_0^{L} e^{-\int_0^x(r+\mu_1(s)\,dx+\lambda_{2})\,ds}\,dx
+ \lambda_{2}\int_0^{L} e^{-\int_0^x(r+\mu_2(s))\,ds}\,dx \nonumber\\
&+\lambda_{2}\lambda_{1}  \int^L_0e^{-\int_0^x(r+\mu_2(s))\,ds}\,dx
\int^{L}_{0} e^{-\int_0^x(r+\lambda_{2}+\mu_1(s))\,ds}\,d x\Big).  \label{EST_Phi_r}
\end{align}
To understand $F_r$ in \eqref{def_F}, we have for any $r\in\mathbb{C}$,
\begin{align}
&\left|\int_0^{L} \mu_i(x)\int^x_0e^{-\int_{\tau}^x(r+\mu_i(s))\,ds}y_i(\tau)\,d\tau\,dx\right|
=\left|\int_0^{L}\left( \int^L_{\tau}\mu_i(x) e^{-\int_{\tau}^x(r+\mu_i(s))\,ds}\,dx\right)\,y_i(\tau)\,d\tau\right|\nonumber\\
&\qquad=\left|\int_0^{L}\left(- \int^L_{\tau}e^{-r(x-\tau)}\,d e^{-\int_{\tau}^x\mu_i(s)\,ds}\right)\,y_i(\tau)\,d\tau\right|\nonumber\\
%&\qquad=\int_0^{L}\left(1-e^{-\int^L_{\tau}(r+\mu_i(s))\,ds} -r \int^L_{\tau} e^{-\int_{\tau}^xr+\mu_i(s)\,ds}\,dx\right)\,y_i(\tau)\,d\tau\\
&\qquad=\left|\int_0^{L}\left(1-r \int^L_{\tau} e^{-\int_{\tau}^xr+\mu_i(s)\,ds}\,dx\right)\,y_i(\tau)\,d\tau\right|\nonumber\\
&\qquad\leq \sup_{\tau\in [0, L]}\left|1-r \int^L_{\tau} e^{-\int_{\tau}^xr+\mu_i(s)\,ds}\,dx\right|\cdot \|y_i\|_{L^1(0, L)} \nonumber\\
&\qquad\leq (1+L |r| e^{L|\mathrm{Re}(r)|}) \|y_i\|_{L^1(0, L)}. \label{EST_F}
\end{align}
Therefore, from \eqref{def_F} and \eqref{EST_F} it is easy to see that $F_{r}\colon X\to \mathbb{R}$ is  a compact integral operator.   Furthermore,  if $\Phi(r)$ is invertible, then 
\begin{align}
p_0=\Phi(r)^{-1}F_r(\vec{y}). \label{oper_p0}
\end{align}
It is clear that for any $r\in \mathbb{C}$, $r\in \rho(\mathcal{A})$ if  and only if $\Phi(r)\neq 0$.  In fact, $\Phi(r)$  is an analytic function defined on the complex plane 
  $\mathbb{C}$ and thus there are  at most countable isolated zeros of $\Phi(r)$. 
Moreover,  from \eqref{EST_Phi_r} it is  easy to verify that $\Phi(r)\neq 0$ if $\mathrm{Re}(r)>0$, and zero is a simple root of $\Phi(r)$ and the only real root, which implies that zero is a simple eigenvalue of $\mathcal{A}$ and the only real eigenvalue of $\mathcal{A}$. In addtion, from \eqref{oper_p1}--\eqref{oper_p2} it is easy to see that the resolvent operator $\mathcal{R}(r, \mathcal{A})$ is a Volterra type of integral operators on $X$, and hence it is compact for any $r\in \rho(\mathcal{A})$. 
 We now have established the statements (2)--(3) and the first part of (1).  It remains to show that there is no other spectra on the imaginary axis except zero.  The proof is elementary yet takes some space. We will leave it in Appendix \ref{app}. 
\end{proof}

With Proposition  \ref{prop1} at our disposal, we are in a position to establish  the  well-posedness of the system~\eqref{IVP}
using Phillips Theorem (e.g.~\cite[Theorem 2.1]{phillips1962semi}).
\begin{thm}\label{thm1}
The system operator $\mathcal{A}$ with its domain $D(\mathcal{A})$ defined in \eqref{def_A}--\eqref{D_A} generates a positive  $C_{0}$-semigroup  of contraction on $X$, denoted by $\mathcal{T}(t)=e^{\mathcal{A}t}, t\geq0$.
\end{thm}
\begin{proof}
According to   Phillips Theorem 
(e.g.~\cite[Theorem 2.1]{phillips1962semi}), it suffices  to show  that (1) $D(\mathcal{A})$  is dense in $X$;
(2) The range $R(I- \mathcal{A}) = X$; and (3) $\mathcal{A}$ is dispersive. 

(1) To show that $\overline{D(\mathcal{A})}=X$, we first let 
\[S=\mathbb{R}\times C^\infty_0[0,L]\times C^\infty_0[0,L]. \]
Since $\overline{S}=X$, it suffices  to prove that $S\subset \overline{\mathcal{D}(\mathcal{A})} $. Let $\vec{y}=(y_0, y_1,y_2)^T \in S$ and consider the sequence $\{\vec{p}_n = (p_{0,n}, p_{1,n}, p_{2,n})^T\}_{n\geq 1}$,  where $p_{0,n}=y_0$, 
\begin{align*}
p_{1,n}(x) :=& \begin{cases}
\lambda_1y_0(1 - nx)^2 + y_1(x), & \text{for } x \in [0, \frac{1}{n}), \\
y_1(x), & \text{for } x \in [\frac{1}{n}, L],
\end{cases}
\end{align*}
and
\begin{align*}
p_{2,n}(x) :=& \begin{cases}
(\lambda_{2}y_0 + \lambda_{2}\int_0^Ly_1(x)dx)(1 - nx)^2 + y_2(x), & \text{for } x \in [0, \frac{1}{n}), \\
y_2(x), & \text{for } x \in [\frac{1}{n}, L].
\end{cases}
\end{align*}
It is clear  that $\vec{p}_n \in \mathcal{D}(\mathcal{A})$  based on the assumption \eqref{mu3}, for all $n \in \mathbb{N}$. Moreover,
\begin{align*}
\|\vec{p}_n - \vec{y}\|_X &= |p_{0,n}-y_0|+\int^L_0|p_{1,n}(x)-y_{1}(x)|\,dx+\int^L_0|p_{2,n}(x)-y_{2}(x)|\,dx\\
&\leq \int^{1/n}_0|\lambda_1y_0(1 - nx)^2|\,dx +\int^{1/n}_0|(\lambda_{2}y_0 + \lambda_{2}\int_0^Ly_1(x)dx)(1 - nx)^2|\,dx\\
&\leq (\lambda_1|y_0| + |(\lambda_{2}y_0 + \lambda_{2}\int_0^Ly_1(x)dx)|)\frac{1}{3n}
\end{align*}
holds for all $n \in \mathbb{N}$. Therefore, $S\subset \overline{\mathcal{D}(\mathcal{A})} $, and hence
 $X=\overline{S}\subseteq \overline{\mathcal{D}(\mathcal{A})} \subseteq X$. It follows that  $\overline{\mathcal{D}(\mathcal{A})} = X$.

(2) From Proposition~\ref{prop1},  we have $1\in\rho(\mathcal{A})$ and thus, for any $\vec{y}\in X$, there exist a unique $\vec{p}\in\mathcal{D}(\mathcal{A})$ such that $(I-\mathcal{A})\vec{p}=\vec{y}$.    

(3) It remains to show that $\mathcal{A}$ is  dispersive. For  $\vec{p}=(p_0, p_1, p_2)^T\in\mathcal{D}(\mathcal{A})$, let
   \begin{align}\label{eq: dispersive map}
 \vec{q}(x)=\left(\frac{[p_0]^+}{p_0}, \frac{[p_1(x)]^+}{p_1(x)}, \frac{[p_2(x)]^+}{p_2(x)} \right),%\in X^*=\mathbb{R}\times L^\infty(0, L)\times L^\infty(0, L), 
 \end{align}
where
  \[  [p_0]^+ =
  \begin{cases}
    p_0,      & \quad \text{if } p_0>0,\\
    0,  & \quad \text{if } p_0\leq 0; \quad
      \end{cases}
        [p_i(x)]^+ =
  \begin{cases}
    p_i(x),     & \quad \text{if } p_i(x)>0,\\
    0,  & \quad \text{if } p_i(x)\leq 0, \quad i=1,2.
  \end{cases}
\]
 Note that the dual space of $X$ is give by $X^*=\mathbb{R}\times L^\infty(0, L)\times L^\infty(0, L)$ and $ \Psi\in X^*$.  Then the duality pairing  
 \begin{align*}
 \langle \mathcal{A}\vec{p}, \vec{q}\rangle=& \left (-\sum^{2}_{i=1}\lambda_i p_0 + \sum_{i=1}^2\int_0^{L} \mu_i(x)p_i(x) \,dx\right)\frac{[p_0]^+}{p_0} \\
 &-\int^L_0\left(\frac{dp_1(x)}{dx}+\mu_{1}(x)p_1(x)+\lambda_{2}p_1(x)\right) \frac{[p_1(x)]^+}{p_1(x)}dx\\
 &-\int^L_0\left(\frac{d{p_2(x)}}{d{x}}+\mu_{2}(x)p_2(x) \right) \frac{[p_2(x)]^+}{p_2(x)}dx.
   \end{align*}
  Let $W_i= \{x \in [0,L] : p_1(x) > 0\}$ and  $W^c_i= \{x \in [0,L] : p_1(x) \leq 0\}$ for $i=1,2$. Then 
  \begin{align*}
  \int^L_0 \frac{d{p_i(x)}}{d{x}}\frac{[p_i(x)]^+}{p_i(x)}\,dx
  = &\int_{W_i} \frac{d{p_i(x)}}{d{x}}\frac{[p_i(x)]^+}{p_i(x)}\,dx
  + \int_{W^c_i} \frac{d{p_i(x)}}{d{x}}\frac{[p_i(x)]^+}{p_i(x)}\,dx\\
  =&\int_{W_i} \frac{d{p_i(x)}}{d{x}}\frac{[p_i(x)]^+}{p_i(x)}\,dx\\
  =&\int^L_{0} \frac{d{[p_i]^+}}{d{x}}\,dx=[p_i(L)]^+-[p_i(0)]^+.
  \end{align*}
 Recall that $\mu_i(x)\geq 0, i=1,2$, for $x\in [0, L]$. Thus 
     \begin{align*}
\langle \mathcal{A}\vec{p}, \vec{q}\rangle = &-\sum^{2}_{i=1}\lambda _i[p_0]^+ 
+\sum_{i=1}^2\int_0^{L} \mu_i(x)p_i(x) dx\frac{[p_0]^+}{p_0}
-[p_1(L)]^++[p_1(0)]^+\\
 &-\int_0^L\left(\mu_{1}(x)[p_1(x)]^++\lambda_{2}[p_1(x)]^+\right)dx\\
 &-[p_2(L)]^++ [p_2(0)]^+ -\int_0^L \mu_{2}(x)[p_2(x)]^+dx \\
 \leq  &-\sum^{2}_{i=1}\lambda _i[p_0]^+ 
+\sum_{i=1}^2\int_0^{L} \mu_i(x)[p_i(x)]^+\, dx
-[p_1(L)]^++\lambda_1[p_0]^+\\
 &-\int_0^L\left(\mu_{1}(x)[p_1(x)]^++\lambda_{2}[p_1(x)]^+\right)dx\\
 &-[p_2(L)]^++ \lambda_{2} [p_0]^++ \lambda_{2}\int^{L}_{0} [p_{1}(x)]^+\, dx  -\int_0^L \mu_{2}(x)[p_2(x)]^+dx\\
 = &-\sum_{i=1}^2[p_i(L)]^+\leq 0,
  \end{align*}
which indicates that $\mathcal{A}$ is dispersive and this completes the proof.
\end{proof}
With the help of Theorem \ref{thm1}, the following results hold immediately. 
\begin{corollary}\label{cor1}
  For a non-negative initial datum  $\vec{p}_0=(\phi_0, \phi_1, \phi_2)^T\in X$, there exists a unique and non-negative  solution given by $\vec{p}(\cdot, t)=\mathcal{T}(t)\vec{p}_0$ to the Cauchy problem \eqref{IVP}.
  Moreover, if $\|\vec{p}_0\|_{X}=1$, then
  \begin{align*}
 \| \vec{p}(\cdot, t)\|_{X}\leq 1, \quad \forall t\geq0.
\end{align*}
\end{corollary}

%%====compact resolvent=======
%With the help of \eqref{oper_p1}--\eqref{oper_p0}, we have
%\begin{align}
%(\mathcal{R}(r, A) \vec{y})(x)=&\Big(\phi(r)^{-1}F(\vec{y}), \lambda_1\phi(r)^{-1}F(\vec{y}) e^{-\int_0^x(r+ds\mu_1(s)+\lambda_{2})\,ds}+\Psi_1(r)y_1,\\
%&\lambda_{2} \phi(r)^{-1}F(\vec{y})e^{-\int_0^x(r+\mu_2(\alpha))\,ds}(1+\lambda_{1} \int^{L}_{0} 
%e^{-\int_0^x(r+ds\mu_1(s)+\lambda_{2})\,ds}\,d x)\nonumber\\
%&+\lambda_{2}e^{-\int_0^x(r+\mu_2(\alpha))ds}\int^{L}_{0}\Psi_1(r)y_1\, dx+\Psi_2(r)y_2\Big).
%\end{align}
%%====compact resolvent=======
%
%
Next we address the asymptotic behavior  of system \eqref{IVP}. According to Proposition \ref{prop1} (3), we know that zero is a simple eigenvalue of $\mathcal{A}$, 
thus setting $$\mathcal{A}\vec{p}=0,$$ one can easily solve the  unique steady-state solution of  system \eqref{IVP}, which is the eigenfunction  corresponding to the eigenvalue zero. Denote it by
$
\vec{p}_{e}(x)=(p_{e0},p_{e1}(x),p_{e2}(x)),
$
where
\begin{align}
&p_{e1}(x)=\lambda_1p_{e0}e^{-\int^{x}_{0}(\lambda_{2}+\mu_1(s))\,ds}, \label{p_e1}\\
&p_{e2}(x) =\lambda_{2}p_{e0}e^{-\int^{x}_{0}\mu_2(s)\, ds}+\lambda_{1}\lambda_2p_{e0}e^{-\int^{x}_{0}\mu_2(s)\, ds}
\int^{L}_{0} e^{-\int^{x}_{0}(\lambda_{2}+\mu_1(s))\,ds}\, dx,\label{p_e2}
\end{align}
and  $p_{e0}$ satisfies 
\begin{align}
p_{e0}=\frac{1}{(1+\lambda_1C_0)(1+\lambda_{2} C_1)}\label{p_e0}
\end{align}
with $C_0=\int^T_0e^{-\int^{x}_{0}(\lambda_{2}+\mu_1(s))\,ds}\,dx$ and $C_1=\int^L_0e^{-\int^{x}_{0}\mu_2(s)\, ds}\,dx$.

%With the help of Proposition  \ref{prop1}, Theorem \ref{thm1}, and \cite[Thm. 2.4]{} we can obtain the asymptotic stability of time-dependent solution of the system.
%===strong stability=====
%\begin{corollary}\label{prop: steady state}
%Consider~\eqref{eq: Cauchy Problem}-\eqref{eq:ini_con}, then we have that
%\[
%\lim_{t\rightarrow\infty}\|\mathcal{T}(t)\vec{p}(0)-\vec{p}_{ss}\|_{X}=0,
%\]
%where $p_{e0}=1$.
%\end{corollary}
%The following result establishes sufficient conditions for exponentially converges. We state the result and refer the reader to~\cite[pg~331, Cor~3.3]{engel2000one} for proof.
%===strong stability=====

When the   maximum repair time $L$ is finite,  we can  further  show  that the semigroup $\mathcal{T}(t)$ is eventually compact, i.e.,  there exists $t_0>0$ such that 
$\mathcal{T}(t_0)$ is compact.  Consequently,  we can establish the exponential stability result. 

\begin{prop}\label{prop2}
The $C_0$-semigroup $\mathcal{T}(t)$  is compact on $X$ when  $t>2L$.
\end{prop}
\begin{proof}
Since the resolvent operator $\mathcal{R}(r,\A)$ for $r\in \rho(\A)$ is compact based on Proposition \ref{prop1} (2), it remains  to  show that $\mathcal{T}(t)$ is continuous in the uniform operator topology for $t>2L$ by  \cite[Cor.\,3.4, p.\,50]{pazy1983semigoroups}, that is, 
\begin{equation}
\lim_{h\to0}\|\mathcal{T}(t+h)-\mathcal{T}(t)\|_{\mathcal{L}(X)}\to 0,  \label{EST_T}
\end{equation}
uniformly  as $h\to0$.  
%Since $X_1:=(D(\mathcal{A}), \|\cdot\|_{\mathcal{A}})\subset \mathbb{R}\times W^{1,1}(0, L)\times W^{1,1}(0, L) \hookrightarrow X$ is compact,  where $\|\psi\|_{\mathcal{A}}=\|\mathcal{A}\psi\|_{X}$ for $\psi\in D(\mathcal{A})$,  we know that $\A$ has compact resolvent  (e.g.~\cite[Prop.\,4,25, p.\,117]{engel2000one}). This can be also shown using the compactness of the integral operator  $F_r$ defined in \eqref{def_F} in Appendix \ref{app}.
  To prove \eqref{EST_T}, we first solve the equations  \eqref{system}--\eqref{sys_IC}   using the method of  characteristics and obtain
   that 
   $$\mathcal{T}(t)\vec{p}_0=(p_0(t),p_1(\cdot,t),p_2(\cdot,t))^{T},$$
   for $\vec{p}_0=(\phi_0, \phi_1, \phi_2)^T\in X$, where
   \begin{equation}
p_0(t)=e^{-\sum_{i=1}^2\lambda_i  t} \phi_{0} +\sum^2_{i=1}\int_0^t\int_0^{L}e^{-\sum_{i=1}^2\lambda_i(t-\tau)}\mu_i(x)p_i(x,\tau)\, dx d\tau,    \label{sol_p0}
\end{equation}

\begin{equation}
p_1(x,t)=\left\{\begin{array}{l}
%\displaystyle\dot{\omega}(t)=-\mu(t+\xi)\omega(t),\\
\displaystyle \lambda_1 p_0(t-x) e^{-\int_{0}^x(\mu_1(\tau)+\lambda_{2})d\tau} ,\quad\text{ if $t>x$},\\
\displaystyle \phi_{1}(x-t) e^{-\int_{x-t}^x(\mu_1(\tau)+\lambda_{2})d\tau},\quad\ \text{ if $ t\leq x$},
\end{array}\right.
\label{sol_p1}
\end{equation}
and
\begin{equation}
p_2(x,t)=\left\{\begin{array}{l}
\displaystyle \left(\lambda_{2} p_0(t-x)+\lambda_{1}\lambda_{2}\int^{L}_{0}  e^{-\int_{0}^x(\mu_1(\tau)+\lambda_{2})d\tau}p_0(t-2x)\, dx\right)e^{-\int_{0}^x\mu_2(\tau)d\tau},\ \text{ if $t>2x$},\\
\displaystyle \left(\lambda_{2} p_0(t-x)+\lambda_{2}\int^{L}_{0} e^{-\int_{2x-t}^x(\mu_1(\tau)+\lambda_{2})d\tau}\phi_{1}(2x-t)\, dx \right)e^{-\int_{0}^x\mu_2(\tau)d\tau},\\
\hspace{4.5in}
\text{ if $x<t\leq 2x$},\\
\displaystyle \phi_{2}(x-t)e^{-\int_{x-t}^x\mu_2(\tau)d\tau},\quad\text{ if $ t\leq x$}.
\end{array}\right.
\label{sol_p2}
\end{equation} 
To establish  \eqref{EST_T},  we have 
\begin{align}
\|T(t+h)\vec{p}_0-\mathcal{T}(t)\vec{p}_0\|_{X}=&|p_0(t+h)-p_0(t)|+\int^L_0|p_1(x, t+h)-p_1(x, t)|\,dx\nonumber\\
&+\int^L_0|p_2(x, t+h)-p_2(x, t)|\,dx, \label{2EST_T}
\end{align}
for any $\vec{p}_0 \in X$ and $h>0$. From \eqref{sol_p0}, we have
\begin{align}
&|p_0(t+h)-p_0(t)|\leq  \vert  e^{-\sum^{2}_{i=1}\lambda_{i}(t+h)}-e^{-\sum^{2}_{i=1}\lambda_{i}t}\vert \cdot  \vert\phi_{0}(t)\vert \nonumber\\
&\qquad+\sum^{2}_{i=1}\left(\int^{t+h}_0-\int^{t}_0\right)\left(e^{-\sum^{2}_{i=1}\lambda_{i}(t+h-\tau) } \left \vert \int^{L}_{0}\mu_{i}(x)p_{i}(x,\tau)\, dx\right \vert \right)\,d\tau \nonumber\\
&\qquad+\sum^{2}_{i=1}\int^{t}_{0}\left \vert (e^{-\sum^{2}_{i=1} \lambda_{i}(t+h-\tau) } -e^{-\sum^{2}_{i=1}\lambda_{i}(t-\tau) }) \right\vert
\cdot \left \vert \int^{L}_{0}\mu_{i}(x)p_{i}(x,\tau)\,dx \right \vert d\tau\nonumber\\
&\quad\leq  \Big(\sum_{i=1}^2\lambda_i \Big)h \vert\phi_{0}(t)\vert
+\sum^{2}_{i=1} \int^{t+h}_t\left \vert \int^{L}_{0}\mu_{i}(x)p_{i}(x,\tau)\, dx\right \vert \,d\tau \nonumber\\
&\qquad+\sum^{2}_{i=1}\int^{t}_{0}\Big(\sum_{i=1}^2\lambda_i\Big)he^{-\sum^{2}_{i=1}\lambda_{i}(t-\tau)}
\left \vert \int^{L}_{0}\mu_{i}(x)p_{i}(x,\tau)\,dx \right \vert d\tau, \label{EST_diff_p0}
\end{align}
%where
%\begin{align}
% \vert  e^{-\sum^{2}_{i=1}\lambda_{i}(t+h)}-e^{-\sum^{2}_{i=1}\lambda_{i}t}\vert \leq h\Big(\sum_{i=1}^2\lambda_i \Big) e^{-\sum^{2}_{i=1}\lambda_{i}t}, \label{EST_diff_exp}
% \end{align}
where from \eqref{sol_p1} for $t>2L$, we get
\begin{align}
 \int_0^{L}\mu_1(x)p_1(x,\tau)\, dx\,d\tau&\leq\lambda_1 \int_0^{L}|p_0(\tau-x)|e^{-\lambda_1x}\mu_1(x)e^{-\int_{0}^x\mu_1(s)ds} \,dx\,d\tau\nonumber\\
 &\leq  \lambda_1\sup_{t\geq 0}|p_0(t)| \label{EST_mu_p1}
\end{align}
and   from~\eqref{sol_p2} for $t>2L$, we get
\begin{align}
&\int_0^{L}\mu_2(x)p_2(x,\tau)\, dx
\leq\lambda_2\int_0^{L}|p_0(\tau-x)|\mu_2(x)e^{-\int_{0}^x\mu_2(\tau)d\tau}\,dx\nonumber\\
&\qquad\quad+\lambda_1\lambda_{2} \left(\int^{L}_{0} e^{-\int_{0}^x(\mu_1(\tau)+\lambda_{2})d\tau}|p_{0}(\tau-2x)|\, dx\right)
 \cdot\left(\int^L_0\mu_2(x) e^{-\int_{0}^x\mu_2(\tau)d\tau}\,dx\right)\nonumber\\%%
&\qquad\leq \lambda_2\sup_{t\geq 0}|p_0(t)|+\lambda_1\lambda_{2}L\sup_{t\geq 0}|p_0(t)|\nonumber\\%%%
&\qquad= (\lambda_2+\lambda_1\lambda_{2}L)\sup_{t\geq 0}|p_0(t)|. \label{EST_mu_p2}
\end{align}
Combining \eqref{EST_diff_p0} with \eqref{EST_mu_p1}--\eqref{EST_mu_p2} yields
\begin{align}
|p_0(t+h)-p_0(t)|
&\leq  \Big(\sum_{i=1}^2\lambda_i\Big)  h\sup_{t\geq 0}|p_0(t)| 
+ \lambda_1h\sup_{t\geq 0}|p_0(t)| + \left(\lambda_2+\lambda_1\lambda_{2}L\right)h\sup_{t\geq 0}|p_0(t)|\nonumber\\
&\qquad+(1-e^{-\sum^{2}_{i=1}\lambda_{i}t})  \left( \lambda_1+ \lambda_2+\lambda_1\lambda_{2}L\right)h\sup_{t\geq 0}|p_0(t)|  \nonumber\\
&\leq c_0(\lambda_1, \lambda_2, L) h\sup_{t\geq 0}\|\mathcal{T}\vec{p}_0\|_{X}, \label{2EST_diff_p0}
\end{align}
where $c_0(\lambda_1, \lambda_2, L)=3\sum_{i=1}^2\lambda_i +2\lambda_1\lambda_{2}L$.

For the  second term  in~\eqref{2EST_T}, we obtain 
\begin{align}
\int_0^L|p_1(x,t+h)-p_1(x,t)|dx&\leq\int^L_0 \lambda_1e^{-\int_{0}^x(\mu_1(\tau)+\lambda_{2})d\tau}|p_0(t+h-x)-p_0(t-x)|\,dx\nonumber\\
&=\int^t_{t-L}\lambda_1e^{-\int_{0}^{t-s}(\mu_1(\tau)+\lambda_{2})d\tau}|p_0(s+h)-p_0(s)|\,ds\nonumber\\
&\leq   \lambda_1L c_0(\lambda_1, \lambda_2, L) h \sup_{t\geq 0}\|\mathcal{T}(t)\vec{p}_0\|_{X}, \label{EST_diff_p1}
\end{align}
for $t>2L$. Similarly,   we can verify that 
\begin{align}
&\int_0^L|p_2(x,t+h)-p_2(x,t)|\,dx 
% \leq \lambda_2\int_{t-L}^te^{-\int_{0}^{t-s}\mu_2(\tau)d\tau}|p_0(s+h)-p_0(s)|\,ds\\
%&\qquad+2\lambda_1\lambda_{2}\int_0^Le^{-\int_{0}^x\mu_2(\tau)d\tau}dx\int_{t-2L}^te^{-\int^{\frac{t-s}{2}}_{0}(\mu_1(\tau)+\lambda_{2})d\tau}|p_0(s+h)-p_0(s)|\,ds\\
%&\leq\lambda_2 La_0h\sup_{t\geq 0}\|\mathcal{T}\vec{p}_0\|_{X}+2\lambda_1\lambda_{2} L \cdot 2L a_0h\sup_{t\geq 0}\|\mathcal{T}\vec{p}_0\|_{X} \\
\leq   \lambda_2L(1 +4\lambda_1L)c_0(\lambda_1, \lambda_2, L)h \sup_{t\geq 0}\|\mathcal{T}\vec{p}_0\|_{X}.  \label{EST_diff_p2}
\end{align}

As a result of \eqref{2EST_diff_p0}--\eqref{EST_diff_p2} and   $\|\mathcal{T}(t)\|_{\mathcal{L}(X)}\leq 1$ for $t\geq 0$, we have 
\begin{align*}
\|T(t+h)\vec{p}_0-\mathcal{T}(t)\vec{p}_0\|_{X}&\leq c_1(\lambda_1, \lambda_2, L)h\sup_{t\geq 0}\|\mathcal{T}(t)\|_{\mathcal{L}(X)}\|\vec{p}_0\|_X,\\
&\leq c_1(\lambda_1, \lambda_2, L)h\|\vec{p}_0\|_X\to 0,
\end{align*}
uniformly as $h\to 0$, where  the constant   $c_1(\lambda_1, \lambda_2, L)>0$  is independent of $h$.  Therefore,
\eqref{EST_T} holds and this completes the proof.
\end{proof}  
However,  the eventual compactness of $\mathcal{T}(t)$ does not  hold when $L=\infty$.
\begin{remark}\label{rem1}
Based on condition  \eqref{mu2} and \eqref{sol_p0}--\eqref{sol_p2}, it is clear that 
  \begin{align}
p_1(L,t)=p_2(L,t)=0, \quad \forall t>0. \label{cond_L}
\end{align}
which implies  that the probability density distributions of the system in degraded and failure modes become zero
once the repair time reaches its maximum. 
As a result, one can derive that 
$$ \frac{dp_0(t)}{dt}+\sum^2_{i=1}\frac{\partial \int^L_0p_i(x,t)\,dx}{\partial t}=0, \quad \forall t>0.
$$
Therefore,  if $\vec{p}_0=(\phi_0, \phi_1, \phi_2)^T\geq 0$ and $\|\vec{p}_0\|_{X}=1$, then  by  the positivity of the semigroup $\mathcal{T}(t)$, 
we have  $\vec{p}(x,t)=\mathcal{T}(t)\vec{p}_0\geq 0$  and
\begin{align}
\|\vec{p}(\cdot, t)\|_{X}=\|\vec{p}_0\|_{X}=1, \quad \forall t>0. \label{conservation}
\end{align}
 In other words, our  system  \eqref{IVP} is conservative in terms of $\|\cdot\|_X$-norm.
Moreover,  it is easy to verify that 
 \begin{align}
& \vec{p}(\cdot, t)=\mathcal{T}(t)\vec{p}_0\in D(\mathcal{A})\quad \text{for}\quad t>2x, \label{1reg_sol}\\
\text{and} \quad &p_i(x,t)\in W^{1,1}(0, L) \quad\text{for} \quad t\leq 2 x,\quad \text{if} \quad \phi_i(x)\in W^{1,1}(0, L), \ i=1,2.\label{2reg_sol}
  \end{align}
\end{remark}
Finally, by  eventual compactness and the fact that zero is a simple eigenvalue of the generator $\mathcal{A}$ and the only spectrum on the imaginary axis established in Theorem \ref{thm1} and Proposition \ref{prop1}, 
 the exponential  stability result   follows immediately from  (e.g.~\cite[Cor.\,3.2,p.\,330]{engel2000one}).
\begin{thm}\label{thm3}
For $\vec{p}_0\in X$, let $\vec{p}(\cdot, t)= \mathcal{T}(t)\vec{p}_0$ be the solution to the Cauchy problem \eqref{IVP}, then it  converges exponentially to its  steady-state solution $\vec{p}_{e}=(p_{e0}, p_{e1},p_{e2})^T$ given by \eqref{p_e1}--\eqref{p_e0}, that is,
\begin{align}
\Vert \vec{p}(\cdot,t)-\vec{p}_{e}(\cdot)\Vert_{X} \leq  M_{0}e^{-\varepsilon_{0}t},
\end{align}
for some constant  $\varepsilon_{0}>0$ and $M_{0}\geq 1$.
\end{thm}

\section{Bilinear controllability  via time-dependent repair actions}\label{sec_control}
In this section, we will focus on the  investigation of bilinear controllability of  system \eqref{IVP} 
via system repair actions.
 Our main objective is stated as follows. 

\textbf{Problem Statement.}
Given $t_f>0$, let  $\vec{p}_0(\cdot)=(\phi_{0}, \phi_{1}(\cdot),\phi_{2}(\cdot))^T\in X$ and  $\vec{p}^*(\cdot)=(p^*_{0}, p^*_{1}(\cdot),p^*_2(\cdot))^{T}\in X$ be  non-negative  initial and desired states, respectively,  with $\|\vec{p}_0\|_{X}=\|\vec{p}^*\|_{X}=1$.  
 Find a vector of  space-time dependent repair rates $\vec{\mu}(x,t)=(\mu_1(x, t), \mu_2(x,t))^T$  such that the solution $\vec{p}(\cdot,t)$ to  system \eqref{IVP} satisfies 
 $\vec{p}(x,t_f)=\vec{p}^*(x)$ for $x\in [0, L]$.
 
However the desired states can not be arbitrary. Due to the properties of non-negativity and conservation of the system, the desired states should also satisfy these attributes described by \eqref{cond_L}--\eqref{2reg_sol} and  the boundary conditions \eqref{sys_BC_D}--\eqref{sys_BC_F}. In addition, we assume that the desired probability density distribution of the system in degraded and failure modes are strictly decreasing functions. In other words, while under repair, it is not expected that the desired density distributions of these two modes  increase.  Specifically, 
 we assume that 
  \begin{align}
\vec{p}^*=(p^*_0, p^*_1, p^*_2)^T\in \mathbb{R}\times W^{1,1}(0, L)\times W^{1,1}(0, L)\label{1desired_p}
\end{align}
and
\begin{align}
&p^*_1(0)=\lambda p^*_0, \quad
p^*_2(0)=\lambda_{2} p^*_0+\lambda_{2}\int^{L}_{0} p^*_{1}(x)\, dx \label{2desired_p}\\
& p^*_i(L)=0, \quad i=1,2, \label{3desired_p}
\end{align}
where 
\begin{align}
p^*_0=1-\sum^2_{i=1}\int^L_0p^*_i(x)\,dx.\label{desired_p0}
\end{align}
Moreover, 
\begin{align}
p^*_i(x)&\geq 0, \quad i=1,2, \quad \forall x\in [0, L],   \label{4desire_p}\\
  \frac{d p^*_1(x)}{dx}&\leq -\epsilon<0 \quad \text{and}\quad  \frac{d p^*_2(x)}{dx}<0 , \quad \forall x\in (0, L)  \label{5desire_p}
\end{align}
for some $\epsilon>0$.
Observe that if the repair rates are time-independent,  then setting  the steady-state solution given by \eqref{p_e1}--\eqref{p_e0} to be the desired distribution, that is, letting
\begin{align}
&p^*_{0}=p_{e0}\nonumber \\
&p^*_{1}(x)=\lambda_1p_{e0}e^{-\int^{x}_{0}(\mu_1(s)+\lambda_{2})\,ds}, \label{eq_mu1}\\
&p^*_{2}(x) =e^{-\int^{x}_{0}\mu_2(\alpha)\, ds}\left(\lambda_{2}p_{e0}+\lambda_{2}\lambda_1p_{e0}\int^{L}_{0} \left(e^{-\int^{x}_{0}(\mu_1(s)+\lambda_{2})\,ds}\right)\, dx\right),  \label{eq_mu2}
\end{align}
we can obtain from  \eqref{eq_mu1}--\eqref{eq_mu2} that
\begin{equation}
\mu_1(x)=-(\frac{d }{dx}\left(\ln{p^*_{1}(x)}\right)+\lambda_{2})=-\left(\frac{p^*_{1_x}}{p^*_{1}}+\lambda_{2}\right)
 \label{mu1_e}
\end{equation}
%where $p_{f1}(0)=\lambda_1p_{e0}$ 
and  
\begin{equation}
\mu_2(x)=-\frac{d }{dx}\left(\ln{p^*_{2}(x)}\right)=-\frac{p^*_{2_x}}{p^*_{2}},  \label{mu2_e}
\end{equation}
%where
%\[
%p_{f2}(0)=\left(\lambda_{2}p_{e0}+\lambda_{2}\int^{L}_{0}p_{f1}(x) dx\right).
%\]
which satisfy  \eqref{mu1}--\eqref{mu2}. This observation implies that if the repair 
rates are set as  \eqref{mu1_e}--\eqref{mu2_e}, then the system solution converges
to $\vec{p}^*$ exponentially. 
  On the other hand, one may possibly adjust  the repair rates in time to steer the system to such a state at some final time. 
  This motivates us to consider the space-time dependent repair rate design in the following section. 
  \subsection{Bilinear Control Design of the repair actions }
   Now  we investigate the bilinear controllability of the repair actions when they are allowed to depend on system running time $t$. Note that for any $t_f>0$ we can always choose a constant $r_0>0$ such that 
  $r_0\sum^{\infty}_{k=1} \frac{1}{k^2}=t_f$.  In fact,  $\sum^{\infty}_{k=1} \frac{1}{k^2}=\frac{\pi^2}{6}$ and hence $r_0=\frac{6t_f}{\pi^2}$. %Without loss of generality, we let $t_f=\sum^{\infty}_{k=1} \frac{1}{k^2}$.  
 Inspired by \cite{elamvazhuthi2017controllability, elamvazhuthi2018bilinear}, we consider
 the space-time dependent repair rates $\mu_i(x, t), i=1,2, $ in the following feedback forms 
 \begin{align}
&\mu_{1}(x,t)=-\frac{1}{p_1(x,t)}\frac{\partial{p_1(x,t)}}{\partial{x}}+\alpha_1j\frac{1}{p_1(x,t)}\frac{\partial{(g_1(x)p_1(x,t))}}{\partial{x}}-\lambda_2, \label{feedback_mu1}\\
&\mu_{2}(x,t)=-\frac{1}{p_2(x,t)}\frac{\partial{p_2(x,t)}}{\partial{x}}+\alpha_2 j\frac{1}{p_2(x,t)}\frac{\partial{(g_2(x)p_2(x,t))}}{\partial{x}}, \label{feedback_mu2}
\end{align}
 for $ t\in [r_0\sum^{j-1}_{k=1} \frac{1}{k^2}, r_0\sum^{j}_{k=1} \frac{1}{k^2})$ and $i\in \mathbb{Z}^{+}$, where $g_i(x)=\frac{1}{p^*_i(x)}, i=1,2$
  and $\alpha_i >0, i=1,2$  are some constants to be properly chosen.   Here   we set $ \sum^{j}_{k=1} \frac{1}{k^2}=0$ if $j=0$ 
  and let $t_j=r_0\sum^{j}_{k=1} \frac{1}{k^2}, j\in \mathbb{Z}^+$, in the rest of our discussion. Observe  that in \eqref{feedback_mu1}--\eqref{feedback_mu2}, $ \mu_i, i =1,2,$ are weighted in time $t$ by $j$ on each interval 
  $[t_j, t_{j+1})$ and it is straightforward to verify that 
$\mu_i\geq 0$  if $\alpha_1 \geq \max\{\frac{\lambda_2}{\epsilon}p^{*2}_1(0),  p^*_1(0)\}$, for all $t\geq 0$ and $\alpha_2 \geq p^*_{2}(0)$.  
%====detailed proof===
In fact, to have $\mu_1(x,t)\geq0$ for $x\in [0, L]$ and $t\geq0$,  we need
 \begin{align*}
\alpha_1j\frac{1}{p_1(x,t)}\frac{\partial{(g_1(x)p_1(x,t))}}{\partial{x}}\geq \frac{1}{p_1(x,t)}\frac{\partial{p_1(x,t)}}{\partial{x}}+\lambda_2,\\
\end{align*}
where
$$\alpha_1j\frac{1}{p_1(x,t)}\frac{\partial{(g_1(x)p_1(x,t))}}{\partial{x}}%=\alpha_1 j\frac{g_x(x,t) p_1(x,t)+g(x,t)p_{1_x}(x,t)}{p_1(x,t)}
=\alpha_1 j g_{1_x}(x)+\alpha_1 j g_1(x)\frac{p_{1_x}(x,t)}{p_1(x,t)}$$
and
$$g_{1_x}(x)=-\frac{p^*_{1_x}(x)}{p^{*2}_1(x)}\geq \epsilon >0.$$
Thus it suffices to have
$\alpha_1 j g_{1_x}(x)\geq \lambda_2$
and
$\alpha_1 jg(x)\frac{p_{1_x}(x,t)}{p_1(x,t)}\geq\frac{p_{1_x}(x,t)}{p_1(x,t)}.$
Since $j\geq 1$, it suffices to choose    $\alpha_1 \geq \max\{ \sup_{x\in[0, L]}\frac{\lambda_2}{g_{1_x}(x)},  \sup_{x\in[0, L]}\ \frac{1}{g_1(x)}\}$, 
i.e.,
$$\alpha_1 \geq \max\{\sup_{x\in[0, L]}\big(- \lambda_2\frac{p^{*2}_1(x)}{p^*_{1_x}}\big),  \sup_{x\in[0, L]} p^*_1(x)\},$$
therefore, we take 
\begin{align}
\alpha_1\geq \max\{\frac{\lambda_2}{\epsilon} p^{*2}_1(0),  p^*_1(0)\}. \label{EST_alpha1}
\end{align}
Similarly, we set
\begin{align}
\alpha_2\geq p^*_2(0). \label{EST_alpha2}
\end{align}

The assumptions \eqref{mu1}--\eqref{mu2} will be verified for  after investigating the properties of the closed-loop system.
The following two theorems establish  the main controllability results of this work.
%====detailed proof===
\begin{thm}\label{thm_main1}
Given $t_f>0$, let  $\vec{p}_0=(\phi_{0}, \phi_{1}(\cdot),\phi_{2}(\cdot))^T\in X$ and  $\vec{p}^*(\cdot)=(p^*_{0}, p^*_{1}(\cdot),p^*_2(\cdot))^{T}\in X$ be  non-negative  initial and desired states, respectively,  with $\|\vec{p}_0\|_{X}=\|\vec{p}^*\|_{X}=1$.  Assume that $\vec{p}^*$ satisfies \eqref{1desired_p}--\eqref{5desire_p}. Then
there exists a vector of  repair rates $\vec{\mu}(x,t)=(\mu_1(x,t),\mu_2(x,t))^T$ in the feedback forms given by \eqref{feedback_mu1}--\eqref{feedback_mu2},
such that  the  solution  $\vec{p}(\cdot, t)$ to our system \eqref{IVP} satisfies $\vec{p}(\cdot, t_f)=(p^*_0, p^*_1(\cdot ), p^*_2(\cdot ))^T$.
\end{thm}
Based on the feedback control designs \eqref{feedback_mu1}--\eqref{feedback_mu2},  a natural question is that wether 
$\vec{\mu}(x,t)$ stays bounded  as $j\to\infty$, i.e., $t\to t_f$, where  $x\in [0, l]$ with $0<l<L$. Our answer is affirmative under appropriate  conditions. 
\begin{thm}\label{thm_main2}
Let   
\begin{align}
 \alpha_1\geq \max\{\frac{\lambda_2}{\epsilon} p^{*2}_1(0),  p^*_1(0), \frac{1}{r_0}\}\quad \text{and }\quad
\alpha_2\geq \max\{p^*_2(0), \frac{1}{r_0}\}. \label{new_alpha}
\end{align}
If   $p^*_{i}(x)\in W^{1, \infty}(0, L)$, $i=1,2$, and $t_f > 2\sum^2_{i=1} \tilde{p}^*_{i}(L)$, then the feedback control law $\vec{\mu}(x, t)$ is bounded  for $x\in [0, l]$ with $0<l<L$ and $0\leq t\leq t_f$.
\end{thm}

 \subsection{Proofs of Theorems \ref{thm_main1}--\ref{thm_main2}}
The proofs of Theorems \ref{thm_main1}-- \ref{thm_main2} mainly utilize  the exponential convergence of  the closed-loop system   to its steady-state. To start with,  incorporating the  feedback laws in \eqref{feedback_mu1}--\eqref{feedback_mu2},  we obtain the closed-loop system as follows
\begin{eqnarray}
\label{closed_loop}
\left\{\begin{array}{l}
\displaystyle\frac{d p_{0,j}(t)}{d t} =j\sum^2_{i=1}\alpha_i \int_0^{L} \frac{\partial{(g_i(x)p_{i, j}(x,t))}}{\partial{x}}\,dx, \\
\displaystyle\frac{\partial{p_{i,j}(x,t)}}{\partial{t}} =-j\alpha_i\frac{\partial{(g_i(x)p_{i,j}(x,t))}}{\partial{x}}, \quad i=1,2,\\
%\displaystyle\frac{\partial{p_{2,j}(x,t)}}{\partial{t}} =-j\alpha_2\frac{\partial{(g_2(x)p_{2,j}(x,t))}}{\partial{x}},
\end{array}\right. 
\end{eqnarray}
where  $p_{0,j}(t):=p_{0}(t)$ and $p_{i,j}(x,t):=p_i(x,t)$ for $i=1,2$,  and $(x,t)\in(0,L)\times(t_{j-1}, t_j)$ for $j\in \mathbb{Z}^+$,  
with  boundary conditions 
\begin{align}
&p_{1,j}(0,t)=\lambda_1 p_{0, j}(t), \label{closed_loop_BC1}\\
& p_{2,j}(0,t)=\lambda_{2} p_{0,j}(t)+\lambda_{2}\int^{L}_{0} p_{1,j}(x, t)\, dx,\label{closed_loop_BC2}
\end{align}
and the initial conditions 
\begin{align}
 &p_{0,j}(0)=p_{0,j-1}(t_{j-1}), \quad p_{i,j}(x,0)=p_{i,j-1}(x, t_{j-1}), \quad i=1,2.
 \label{closed_loop_IC} 
\end{align}

We first apply the method of characteristics to analyze the transport equations in  \eqref{closed_loop}. To this end, we let $\varphi_{0, j}(t)=p_{0, j}(t)$ and $\varphi_{i, j}(x,t)=j \alpha_i g_i(x)p_{i, j}(x, t)$ for $ t\in (t_{j-1}, t_j)$ and $j\in \mathbb{Z}^+$. Then  the closed-loop system
 \eqref{closed_loop}--\eqref{closed_loop_IC} becomes
\begin{eqnarray}
\left\{\begin{array}{l}
\displaystyle\frac{d \varphi_{0, j}(t)}{d t}=\sum^2_{i=1} \int_0^{L} \frac{\partial{\varphi_{i, j}}(x,t)}{\partial x} \,d x,\\
\frac{\partial{\varphi_{i, j}(x,t)}}{\partial{t}}  =-j \alpha_i  g_i(x) \frac{\partial \varphi_{i, j}(x,t) }{\partial x},  \quad i=1,2,
\end{array}\right. 
\label{2closed_loop}
\end{eqnarray}
with boundary conditions
\begin{align}
\varphi_{i}(0,t)&=j \alpha_1g_1(0)p_{i, j}(0,t), \quad i=1,2,\label{closed_loop_BC_varphi}
%\varphi_{1}(0,t)&=j \alpha_1g_1(0)p_{1, j}(0,t)=j \alpha_1g_1(0)\lambda_{1}p_{0, j}(t)=j \alpha_1g_1(0)\lambda_1  \varphi_{0, j}(t),\label{closed_loop_BC_varphi1}\\
%\varphi_{2}(0,t)&=j \alpha_2g_2(0)p_{2, j}(0,t)%=j \alpha_2g_2(0)(\lambda_{2} p_{0, j}(t)+\lambda_{2}\int^{L}_{0} p_{1, j}(x, t)\, dx)\\
%&=j \alpha_2 g_2(0)\lambda_{2}\left( \varphi_{0, j}(t)+\frac{1}{j\alpha_1}\int^{L}_{0} \frac{\varphi_{1, j}(x, t)}{g_1(x)}\,dx\right),
%\label{closed_loop_BC_varphi2}
\end{align}
and initial conditions
\begin{align}
\varphi_{0, j}(0)=p_{0, j}(0), \quad
 \varphi_{i, j}(x,0)=j \alpha_i g_i(x)p_{i, j}(x, 0).
\label{closed_loop_IC_varphi}
\end{align}

Let  $\frac{dx}{dt}=j\alpha_i  g_i(x)$ with $x(0)=x_0.$
Then
 $\frac{dx}{\alpha_i  g_i(x)}=\frac{1}{j\alpha_i }p^*_{i}(x)\,d x = dt.$
 Let 
 \begin{align}
 \tilde{p}^*_{i,j }(x)=\frac{1}{j\alpha_i }\int^x_0 p^*_{i}(s)\,ds, i=1,2.\label{cha_CL}
 \end{align}
  Then  $\tilde{p}^*_{i, j}(x)= t+\frac{1}{j\alpha_i }\int^{x_0}_0 p^*_{i}(s)\,ds$. Since
$ \frac{ d\tilde{p}^*_{i, j}}{dx}=\frac{1}{j\alpha_i } p^*_{i}>0$ for $ x\in (0, L)$ 
 by \eqref{4desire_p}, this  implies that  $\tilde{p}^*_{i, j}(x)$ is  a  monotonically  increasing  function for $ x\in [0, L]$, and hence invertible. 
It is worth to point out that by \eqref{cha_CL}, \eqref{5desire_p} and \eqref{EST_alpha1}--\eqref{EST_alpha2} we have
 \begin{align*}
 \tilde{p}^*_{i,j }(L)=\frac{1}{j\alpha_i }\int^L_0 p^*_{i}(s)\,ds\leq L, \quad j\in \mathbb{N}^+,  \quad i=1,2.% \label{EST_cha_CL}
 \end{align*}
 Now let 
$  \xi_{i,j}=\tilde{p}^*_{i, j}(x)- t$. Then $x=(\tilde{p}^*_{i, j})^{-1} (t+\xi_{i,j}). $
 Define $$\Psi_{i}(t)=\varphi_{i}((\tilde{p}^*_{i})^{-1}(  t+\xi_{i,j}), t).$$  Then  we have 
\begin{align}
\frac{d\Psi_{i,j}}{dt}=\frac{\partial \varphi_{i, j}}{\partial t}+ \alpha_i g_i(x) \frac{\partial \varphi_{i,j}}{\partial x}
=0.\label{eq_psi}
\end{align}

For  $\xi_i<0$, i.e., $\tilde{p}^*_{i, j}(x)<  t$, the solution to \eqref{eq_psi} is determined by the boundary conditions \eqref{closed_loop_BC_varphi}, so we integrate \eqref{eq_psi} from some $t$ 
such that $x=(\tilde{p}^{*}_{i,j})^{-1}(t+\xi_{i,j})=0$, i.e., $t+\xi_i=0$, and hence $t=-\xi_{i,j}$. Integrating   \eqref{eq_psi}   from 
$-\xi_{i,j}$ to $t$  follows
\begin{align}
\Psi_{i, j}(t)&=\varphi_{i,j}(0, -\xi_{i,j})=\varphi_{i, j}(0,t -\tilde{p}^*_{i, j}(x))), \quad i=1,2.
%\nonumber\\
%&=j\alpha_1 g_1(0)\lambda_1\phi_{0,j}\left(  t-\tilde{p}^*_{1,j}(x)\right)
 \label{EST_phi}
\end{align}
%and 
%\begin{align}
%\Psi_{2,j}(t)&=\phi_{2,j}(0, -\xi_1)=\phi_{2,j}(0, -\tilde{p}^*_{2,j}(x)+ t))
%\nonumber\\
%&= j\alpha_2 g_2(0)\lambda_{2} \phi_{0, j}(t-\tilde{p}^*_{2,j}(x))+\frac{\alpha_2g_2(0)\lambda_{2}}{\alpha_1}\int^{L}_{0} \frac{\phi_{1}(x, t-\tilde{p}^*_{2,j}(x))}{g_1(x)}\,dx. \label{EST_phi2}
%\end{align}

For  $\xi_{i,j}\geq 0$, i.e., $\tilde{p}^*_{i, j}(x)\geq   t$, the solution to \eqref{eq_psi} is determined by the initial  conditions \eqref{closed_loop_IC_varphi}. So we integrate  \eqref{eq_psi}  from $0$ to $t$  and obtain 
\begin{align}
\Psi_{i, j}(t)&=\varphi_{i}((\tilde{p}^*_{i, j})^{-1}(\xi_{i,j}), 0)\nonumber\\
&=\varphi_{i, j}\big((\tilde{p}^*_{i, j})^{-1}(\tilde{p}^*_{i, j}(x)-  t), 0\big). \label{2EST_phi1}
\end{align}
To simplify the notation, we let $$h_{i,j}(x, t)=(\tilde{p}^*_{i, j})^{-1}(\tilde{p}^*_{i, j}(x)- t), \quad    \tilde{p}^*_{i, j}(x)\geq t. $$
 Therefore, according to the boundary and initial conditions \eqref{closed_loop_BC_varphi}--\eqref{closed_loop_IC_varphi}, we have
 %====generic expression=======
\begin{align*}
\varphi_{i, j}(x, t)=
\left\{\begin{array}{c}
\displaystyle 
j \alpha_{i}g_i(0) p_{i, j}(0, t -\tilde{p}^*_{i, j}(x)), \quad\text{if} \quad t>\tilde{p}^*_{i, j}(x);\\
j \alpha_i g_i(h_{i, j}(x, t))p_{i, j}(h_{i, j}(x,t), 0), \quad \text{if} \quad   t\leq  \tilde{p}^*_{i, j}(x).
 \end{array}\right. 
\end{align*}
 %====generic expression=======
%  %====detailed expression=======
% \begin{align*}
%\varphi_{1,j}(x, t)=
%\left\{\begin{array}{c}
%\displaystyle 
%j \alpha_{1}g_1(0)\lambda_1  \varphi_{0, j}(t -\tilde{p}^*_{1}(x)), \quad\text{if} \quad \tilde{p}^*_{1, j}(x)<t;\\
%j \alpha_1 g_1(h_{1, j}(x, t))p_{1, j}(h_{1, j}(x,t), 0), \quad \text{if} \quad   \tilde{p}^*_{1, j}(x)\geq t.
% \end{array}\right. 
%\end{align*}
%and
% \begin{align*}
%\varphi_{2, j}(x, t)=
%\left\{\begin{array}{c}
%\displaystyle 
%j\alpha_2 g_2(0)\lambda_{2}\left( \varphi_{0}(t -\tilde{p}^*_{2, j}(x))+\frac{1}{j \alpha_1}\int^{L}_{0} \frac{\varphi_{1}(x, t -\tilde{p}^*_{2, j}(x))}{g_1(x)}\,dx\right) , \quad\text{if} \quad \tilde{p}^*_{2, j}(x)<t;\\
%j\alpha_2 g_2(h_{2}(x, t))p_{2, j}(h_{2}(x, t), 0), \quad \text{if} \quad   \tilde{p}^*_{2, j}(x)\geq t.
% \end{array}\right. 
%\end{align*}
%  %====detailed expression=======
Solving $\varphi_{0, j}(t)$ from \eqref{2closed_loop} yields 
%====details=========
%\begin{align*}
%\frac{d\phi_0(t)}{dt}&=\phi_1(L, t)-\phi_1(0, t)\\
%&=\left\{\begin{array}{c}
%\displaystyle 
%\phi_1((\tilde{p}^*_1)^{-1}(\tilde{p}^*_{1,i}(L)-\alpha  t), 0)\\-g(0)\lambda\phi_0, \quad\text{if}\quad \tilde{p}^*_{1,i}(L) \geq \alpha  t,\\
%g(0)\lambda\phi_0(  t-\frac{\tilde{p}^*_{1,i}(L)}{\alpha })\\-g(0)\lambda\phi_0, \quad\text{if}\quad \tilde{p}^*_{1,i}(L)<\alpha t.
% \end{array}\right. 
%\end{align*}
%====details=========
\begin{align*}
&\varphi_{0, j}(t)=\sum^2_{i=1}\int^t_0\big(\varphi_{i}(L,  \tau)-\varphi_{i}(0,  \tau)\big)\,d\tau+\varphi_{0, j}(0).
%&=\left\{\begin{array}{c}
%\displaystyle 
%\sum^2_{i=1}\int^t_0\big(\phi_{1,i}(\psi(L, \tau), 0)
%-\phi_{i,j}(0,t -\tilde{p}^*_{i,j}(x))\big)\,d\tau \\
%+\phi_{0,j}(0), \quad  t\leq \tilde{p}^*_{1,j}(L) ;
%\\
%\vspace{0.1in}
%\alpha i g(0)\lambda \int^t_ 0\big(\phi_{0,i}(  \tau-\tilde{p}^*_{1,i}(L))-\phi_{0,i}(\tau)\big)\, d\tau \\
%+\phi_{0,i}(0), \quad t>\tilde{p}^*_{1,i}(L).
% \end{array}\right. 
\end{align*}
Thus
\begin{align}
p_{0, j}(t)=&\sum^2_{i=1}\int^t_0\big(\varphi_{i, j}(L,  \tau)-\varphi_{i}(0,  \tau)\big)\,d\tau+p_{0, j}(0)\nonumber\\
=&j \sum^2_{i=1}\alpha_i  \int^t_0 \big(g_i(L)p_{i, j}(L, \tau)-g_i(0)p_{i, j}(0, \tau)\big)\,d\tau+p_{0, j}(0), \label{closed_loop_p0}
\end{align}
\begin{align}
p_{1, j}(x, t)= \frac{1}{j \alpha_1 g_1(x)}\varphi_{1}(x,t)=
\left\{\begin{array}{c}
\displaystyle 
\frac{g_1(0)}{  g_1(x)}\lambda_1  p_{0, j}(t -\tilde{p}^*_{1, j}(x)), \quad\text{if} \quad  t>\tilde{p}^*_{1, j}(x);\\
\frac{g_1(h_{1}(x, t))}{  g_1(x)}p_{1, j}(h_{1}(x, t), 0), \quad \text{if} \quad  t\leq \tilde{p}^*_{1, j}(x);
 \end{array}\right. 
  \label{closed_loop_p1}
\end{align}
and
\begin{align}
p_{2, j}(x, t)=& \frac{1}{j\alpha_2 g_2(x)}\varphi_{2}(x,t)
=\left\{\begin{array}{c}
\displaystyle 
 \frac{ g_2(0)}{g_2(x)}\lambda_{2} \left(p_{0}(t-\tilde{p}^*_{2, j}(x))+  \int^{L}_{0} p_1(x,t-\tilde{p}^*_{2, j}(x))\,dx\right), \\
\hspace{1.8in} \text{if} \quad t>\tilde{p}^*_{2,j}(x),\\
\frac{ g_2(h_{2}(x, t))}{g_2(x)}p_{2, j}(h_{2}(x, t), 0), \quad \text{if} \quad  t\leq  \tilde{p}^*_{2, j}(x).\\
 \end{array}\right. 
\nonumber\\
  =&\left\{\begin{array}{c}
\displaystyle 
 \frac{ g_2(0)}{ g_2(x)} \lambda_{2} \left(p_{0, j}(t-\tilde{p}^*_{2, j}(x))
+g_1(0)\lambda_1 \int^{L}_{0} \frac{  p_{0, j}(t -\tilde{p}^*_{1, j}(x)-\tilde{p}^*_{2, j}(x)) }{g_1(x)}\,dx\right), \\
\hspace{2.3in}\text{if} \quad t> \tilde{p}^*_{1, j}(x)+\tilde{p}^*_{2, j}(x);\\
 \frac{ g_2(0)}{g_2(x)}\lambda_{2} \left(p_{0, j}(t-\tilde{p}^*_{2, j}(x))
+ \int^{L}_{0} \frac{g_1(h_{1}(x, t- \tilde{p}^*_{2}(x)))p_{1,j}(h_{1}(x, t- \tilde{p}^*_{2}(x)), 0)}{g_1(x)}\,dx\right), \\
\hspace{3in}\text{if} \quad \tilde{p}^*_{2, j}(x)<t\leq \tilde{p}^*_{1, j}(x)+\tilde{p}^*_{2, j}(x), \\
\frac{ g_2(h_{2}(x, t))}{g_2(x)}p_{2, j}(h_{2, j}(x, t), 0), \quad \text{if} \quad  t\leq \tilde{p}^*_{2, j}(x).
 \end{array}\right. 
  \label{closed_loop_p2}
  \end{align}
  
By virtue of \eqref{closed_loop_p1}--\eqref{closed_loop_p2}, it is easy to verify  that   $\mu_i, i=1,2,$ in the feedback forms \eqref{feedback_mu1}--\eqref{feedback_mu2} satisfy \eqref{mu1}--\eqref{mu2} due to $g_i(x)<\infty$ for $x\in [0, l]$ and $g_i(L)=\infty$.

Since $j$ is a weight parameter of the system in each time interval $(t_{j-1}, t_j), j\in \mathbb{Z}^+$, without loss of generality we first analyze  the  closed-loop system \eqref{closed_loop}--\eqref{closed_loop_IC} for $j=1$. In this case, 
\[\tilde{p}^*_{i, 1}(x)=\frac{1}{\alpha_i }\int^x_0 p^*_{i}(s)\,ds, \quad i=1,2. \]
For simplicity,  we denote $\tilde{p}^*_{i, 1}(x)$ by $\tilde{p}^*_{i}(x), i=1, 2$ in  the following discussions.

Define the closed-loop system operator $\mathcal{A}_c\colon D(\mathcal{A}_c)\subset X \to X$
by
\begin{align}
\mathcal{A}_c= \displaystyle 
\begin{pmatrix}
0 &  \alpha_1\int^L_0\frac{d{(g_1(x)\cdot)}}{dx} dx&\alpha_2\int^L_0\frac{d{(g_2(x)\cdot)}}{dx} dx \\ 
0  & - \alpha_1 \frac{d{(g_1(x)\cdot)}}{dx}&0\\
0&0&- \alpha_2 \frac{d{(g_2(x)\cdot)}}{d x}
\end{pmatrix},
 \label{oper_Ac}
\end{align}
with domain
\begin{align}
D(\mathcal{A}_c)=\big\{\vec{p}=&(p_0, p_1(\cdot), p_2(\cdot))^T\in  X\colon p_i, g_ip_i\in W^{1,1}(0, L),
 \quad i=1,2, \nonumber\\
& \text{and}\quad (p_1(0) \ p_2(0))^T=\Gamma_1 p_1(x)+\int_0^L\Gamma_2 p_2(x)dx \big\},
\label{D_Ac}
\end{align} 
where $\Gamma_i, i=1,2$ are defined in  \eqref{oper_BC}.
Note that  $g_ip_i\in W^{1,1}(0, L)$ implies  $g_ip_i\in C[0, L]$ by Sobolev imbedding, %(e.g.,\cite[p.\,100, Lem.\,5.8]{adams1975sobolev}), 
and hence
\begin{equation}\label{eq: finite at the end point}
\int^L_0\frac{d(g_i(x)p_i(x))}{d x}\,dx=g_i(L)p_i(L)-g_i(0)p_i(0)<\infty.
\end{equation}
We keep the integral forms in our formulation to better demonstrate the structure of the closed-loop system. 
Moreover, since $\lim_{x\to L}g_i(x)=\lim_{x\to L}\frac{1}{p^*_i(x)}=\infty, i=1,2,$ based on the assumption \eqref{3desired_p}, we must have $p_{i}(L)=0, i=1,2$ for $p_i \in D(\mathcal{A})$.

Now the  closed-loop system \eqref{closed_loop} for $j=1$ can be formulated as 
\begin{align}
\left\{\begin{array}{ll}
\dot{\vec{p}}(t)=\mathcal{A}_c\vec{p} (t),\\
\vec{p}_0=(\phi_0,\phi_1,\phi_2)^T.
\end{array}
\right. \label{IVP_closed_loop}
\end{align}
It can be shown that  \eqref{IVP_closed_loop} inherits the properties from the open-loop system \eqref{IVP}. Specifically,  the system is conserved in terms of 
$\|\cdot\|_{X}$-norm. 
Moreover, applying  the similar procedures as in the proof of Theorem \ref{thm1}, we can establish  the well-posedness of  system \eqref{IVP_closed_loop} as follows. 
\begin{prop}\label{prop_closedd_loop}
 The closed-loop system operator $\mathcal{A}_c$ with its domain $D(\mathcal{A}_c)$ defined in \eqref{oper_Ac}--\eqref{D_Ac} generates a positive  $C_{0}$-semigroup  of contraction on $X$. Denote it by
 $\mathcal{T}_c(t)=e^{\mathcal{A}_ct}, t\geq0$. Then there exists a unique solution to system \eqref{closed_loop} given by $\vec{p}(x,t)=(\mathcal{T}_c(t)\vec{p}_0)(x)$ for $\vec{p}_0\in X$.
 \end{prop}

It can be  verified that zero is also  a simple eigenvalue of $\mathcal{A}_c$ and the only spectrum on the imaginary axis as in  Theorem \ref{thm1}. Moreover, 
$$
\vec{p}_c=\Big(p^*_{0}, \frac{1}{g_1(x)}, \frac{1}{g_2(x)}\Big)^T =\Big(p^*_{0}, p^*_1(x), p^*_2(x)\Big)^T
$$
is the  eigenfunction corresponding to zero, where $p^*_0$ satisfies \eqref{desired_p0}.  Furthermore, we can obtain the eventual compactness property of the closed-loop system as well.
 
\begin{prop}\label{prop3}
  The $C_0$-semigroup $\mathcal{T}_c(t)$ is compact on $X$ when $t> 2\sum^2_{i=1}\tilde{p}^*_{i}(L)$.
\end{prop}

\begin{proof}
Following the similar approaches as in the proof of Proposition \ref{prop1}, we can show that
 $\A_c$ has compact resolvent. It remains to show that 
 $\mathcal{T}_c(t)$ is continuous in the uniform operator topology for $t>2\sum^2_{i=1}\tilde{p}^*_{i}(L)$, that is, 
 \begin{equation}
\lim_{h\to0}\|\mathcal{T}_c(t+h)-\mathcal{T}_c(t)\|_{\mathcal{L}(X)}\to 0.  \label{EST_Tc}
\end{equation}
Since $\tilde{p}_{i}(x)$ is  monotonically  increasing for $x\in [0, L]$, if $t\geq \tilde{p}^*_{1}(L)+\tilde{p}^*_{2}(L)$,  from \eqref{closed_loop_p0} we have
\begin{align}  
&|p_{0}(t+h)-p_{0}(t)|=\Big|  \sum^2_{i=1}\alpha_i    \int^{t+h}_t  \big(g_i(L)p_{i}(L, \tau)-g_i(0)p_{i}(0, \tau)\big)\,d\tau\Big|\nonumber \\
&\quad \leq  \alpha_1 g_1(0) \lambda_1 \int^{t+h}_t   |  p_{0}(\tau -\tilde{p}^*_{1}(L))-p_0(0, \tau)|\,d\tau\nonumber \\
&\qquad+ \alpha_2 \lambda_{2}  g_2(0)\int^{t+h}_t  \Big|\left(p_{0}(\tau-\tilde{p}^*_{2}(L))
+ \frac{g_1(0)\lambda_1 }{\alpha_1}\int^{L}_{0} \frac{ p_{0}(\tau -\tilde{p}^*_{1}(x)-\tilde{p}^*_{2}(x)) }{g_1(x)}\,dx\right)\nonumber \\
&\qquad- \Big(p_0(\tau)+\int^T_0p_1(x,\tau)\,dx\Big)\Big|\,d\tau\nonumber \\
&\quad \leq  2  \alpha_1 g_1(0)\lambda_1 h  \sup_{t\geq 0 } |p_0(t)| 
+2\alpha_2g_2(0)\lambda_2h \sup_{t\geq 0 } |p_0(t)| \nonumber\\
&\qquad+\alpha_2g_2(0)\lambda_2h\Big(\frac{g_1(0)\lambda_1}{\alpha_1}\sup_{t\geq 0 } |p_0(t)|  +\sup_{t\geq 0 }\|p_1\|_{L^1}\Big)\nonumber \\
&\quad \leq C_0h   \sup_{t\geq 0 } \|\mathcal{T}_c(t)\vec{p}_0\|_X,
\label{EST_diff_p0_CL}
\end{align}
where $C_0=2  \sum^2_{i=1}\alpha_i g_i(0)\lambda_i+\alpha_2g_2(0)\lambda_2 \max\{\frac{g_1(0)\lambda_1}{\alpha_1},1\} $. 
Moreover, from \eqref{closed_loop_p1}, we have
\begin{align}
&\Big|\int^L_0p_{1}(x, t+h)-p_{1}(x, t)\,dx \Big|
=g_1(0) \lambda_1 \Big| \int^L_{0}\frac{1}{ g_1(x)}  \Big(p_{0}(t +h-\tilde{p}^*_{1}(x)-p_{0}(t -\tilde{p}^*_{1}(x)\Big)\,dx\Big|.
\label{1EST_diff_p1_CL}
\end{align}
Let
$
\tilde{t}=t-\tilde{p}^*_{1}(x)>0,
$
then $x=(\tilde{p}^*_{1})^{-1}(t-\tilde{t})$ and 
$d\tilde{t}=-p^{*}_1(x)dx,$
and hence
 \eqref{1EST_diff_p1_CL} satisfies 
\begin{align}
\Big|\int^L_0p_{1}(x, t+h)-p_{1}(x, t)\,dx \Big|\leq \lambda_1 \int^{t}_{t-\tilde{p}^*_{1}(L)}  \Big| \Big(p_0(\tilde{t}+h)-p_0(\tilde{t})\Big)\Big| \,d\tilde{t}.\label{2EST_diff_p1_CL}
\end{align}
Furthermore, in light of   \eqref{EST_diff_p0_CL} for $t-\tilde{p}^*_{1}(L)>\tilde{p}^*_{1}(L)+\tilde{p}^*_{2}(L)$ we get
\begin{align}
&\Big|\int^L_0p_{1}(x, t+h)-p_{1}(x, t)\,dx \Big|  
\leq C_1 h  \sup_{t\geq 0 } \|\mathcal{T}_c(t)\vec{p}_0\|_X, \label{3EST_diff_p1_CL}
\end{align}
for $C_1=\lambda_1\tilde{p}^*_{1}(L)C_0$. 

Using similar analysis, letting $\tilde{t}=t-2\tilde{p}^*_{1}(x)$ one can verify that 
\begin{align}
&\Big|\int^L_0p_{2}(x, t+h)-p_{2}(x, t)\,dx \Big|
\leq g_2(0)\lambda_2\Big|\int^{L}_{0} \frac{1}{ g_2(x)} \Big(p_0(t+h-\tilde{p}^*_{2}(x))-p_0(t -\tilde{p}^*_{2}(x))\Big) \,dt\Big|\nonumber\\
&\qquad+g_2(0)\lambda_2 g_1(0)\lambda_1\Big|  \int^{L}_{0} \int^{L}_{0} \frac{  p_{0}(t+h -\tilde{p}^*_{1}(x)-\tilde{p}^*_{2}(x))
- p_{0}(t -\tilde{p}^*_{1}(x)-\tilde{p}^*_{2}(x)) }{g_1(x)}\,dx\,dx\Big|\nonumber\\
 &\quad\leq \lambda_2\int^{t}_{t-\tilde{p}^*_{2}(L)}\Big| p_0(\tilde{t}+h)-p_0(\tilde{t})\Big| \,d\tilde{t}
 +g_2(0)\lambda_2 \lambda_1 L \int^{t}_{t-\tilde{p}^*_{1}(L)-\tilde{p}^*_{2}(L)} \Big| p_{0}(\tilde{t}+h)- p_{0}(\tilde{t} )\Big|\,dx\label{EST_diff_p2_CL}   \\
&\quad \leq C_2 h  \sup_{t\geq 0 } \|\mathcal{T}_c(t)\vec{p}_0\|_X,\label{2EST_diff_p2_CL}   
\end{align}
for some constant $C_2>0$, where from \eqref{EST_diff_p2_CL} to \eqref{2EST_diff_p2_CL} we need  $t> 2 \tilde{p}^*_{1}(L)+2\tilde{p}^*_{2}(L)$ in order to apply  the estimate in \eqref{EST_diff_p0_CL}.

 As a result of   \eqref{EST_diff_p0_CL}--\eqref{2EST_diff_p2_CL} we have that for $t>  2\sum^2_{i=1}\tilde{p}^*_{i}(L)$,
\begin{align*}
\|\mathcal{T}_c(t+h)\vec{p}_0-T_c(t)\vec{p}_0\|_{X}\leq Ch \sup_{t\geq 0 } \|\mathcal{T}_c(t)\vec{p}_0\|_X
\leq Ch \|\vec{p}_0\|_X \to 0
\end{align*}
 uniformly as $h\to 0$ for any $\vec{p}_0\in $X, where $C>0$ is a constant independent of $h$.
This completes the proof.
\end{proof}

\begin{corollary}\label{cor2}
For $\vec{p}_0\in X$, let $\vec{p}(\cdot, t)= \mathcal{T}_c(t)\vec{p}_0$ be the solution to the closed-loop system  \eqref{IVP_closed_loop}, then it  converges exponentially to its  steady-state solution $\vec{p}^*=(p^*_{0}, p^*_{1},p^*_{2})^T$ satisfying \eqref{1desired_p}--\eqref{5desire_p}, that is,
\begin{align}
\Vert \vec{p}(\cdot,t)-\vec{p}^*(\cdot)\Vert_{X} \leq  M_{c}e^{-\varepsilon_{c}t},
\end{align}
for some constant  $\varepsilon_{c}=\varepsilon_{c}(\alpha_1, \alpha_2)>0$ and $M_{c}\geq 1$.
\end{corollary}
Note that the decay rate $\varepsilon_c$ depends on $\alpha_1$ and $\alpha_2$. One can increase $\varepsilon_c$ by increasing both
of these parameters. With the result of Corollary \ref{cor2}, we are ready to prove our main Theorems \ref{thm_main1}--\ref{thm_main2}.
\begin{proof}[Proof of Theorem \ref{thm_main1}]
According to  \eqref{closed_loop}, the closed-loop system is now weighted by $ j$ for $t\in [t_{j-1}, t_j), j\in \mathbb{Z}^+$, and hence the decay rate of the system solution  to its steady-state
becomes $ j\varepsilon_c$ for $t\in [t_{j-1}, t_j)$. Further note that $t_j-t_{j-1}=\frac{r_0}{j^2}$. Consequently, 
by Corollary~\ref{cor2} we have
\begin{align}
&\|\vec{p}(\cdot, r_0 \sum^{j}_{k=1}\frac{1}{k^2} )-\vec{p}^*(\cdot)\|_{X}
\leq M_ce^{-\sum^{j}_{k=1}   k \varepsilon_{c} \frac{r_0}{k^2} }
\leq M_ce^{- \varepsilon_{c}r_0  \sum^{j}_{k=1} \frac{1}{k}}. \label{EST_key}
\end{align}
 Since $t_f=r_0\sum^{\infty}_{k=1} \frac{1}{k^2}$ and  $\lim_{j\to \infty}\sum^{j}_{k=1} \frac{1}{k}$ diverges, we conclude that
 $$ \vec{p}(\cdot, t_f)=\vec{p}^*(x),$$
 which completes the proof.
\end{proof}

To show the boundedness of the feedback law $\vec{\mu}(x,t)$ as stated in Theorem \ref{thm_main2},
we first note that $p_{i,j}(x,t)\to p^*_{i}(x)$ as $j\to \infty$, where  $p^*_{i}(x)$ is strictly positive and bounded for $x\in [0, l]$. Thus $p_{i, j}(x,t)$ is strictly positive in $[0, l]$ for $j$ sufficiently large. Moreover,  if $p^*_{1}\in W^{1, \infty}(0, L)$, then $\frac{p_{{i,j}_x}(x,t)}{p_{i,j}(x,t)}$ converges to $\frac{p^*_{i_x}}{p^*_i(x)}, i=1,2,$ as $j\to \infty$, which are in $L^\infty[0, l]$.
Therefore, from \eqref{feedback_mu1}--\eqref{feedback_mu2} it suffices to show that 
$j \alpha_i \sup_{x\in [0, l]}  |\frac{\partial(  g_i(x)p_i(x, t))}{\partial x} |$ is finite  for $j\in \mathbb{Z}^+$ sufficiently large.  To this end, we  establish  the following result.
%\begin{comment}
\begin{prop}\label{prop4}
Let  $\alpha_1$ and $\alpha_2$  satisfy  \eqref{new_alpha}.  For $t_f > 2\sum^2_{i=1} \tilde{p}^*_{i}(L)$,  the solution 
$\vec{p}(\cdot, t)=(p_{0}(t), p_{1}(\cdot, t), p_2(\cdot, t)^T$ to the closed-loop system \eqref{IVP_closed_loop} satisfies:
\begin{align*}
&\lim_{j\to \infty}\Big( j \alpha_i\sup_{x\in  [0, l]} \Big|\frac{\partial(  g_i(x)p_i(x, t))}{\partial x}\Big|\Big) <\infty, \quad i=1,2.
\end{align*}
\end{prop}
\begin{proof}

For $t_f > 2\sum^2_{i=1} \tilde{p}^*_{i}(L)$, there exists a $j\in\mathbb{Z}^+$ large enough such that $t_{j-1}\geq 2\sum^2_{i=1} \tilde{p}^*_{i}(L)$. 
Thus for $t> t_{j-1}$ we have 
$t>   2\sum^2_{i=1} \tilde{p}^*_{i}(L)$, 
and hence by \eqref{closed_loop_p1}--\eqref{closed_loop_p2}, 
\begin{align*}
p_{1,j}(x, t)=&\frac{g_1(0)}{ g_1(x)}\lambda_1  p_{0}(t -\tilde{p}^*_{1, j}(x)),\\
p_{2,j}(x, t)=& \frac{ g_2(0)}{g_2(x)} \lambda_{2} \left(p_{0}(t-\tilde{p}^*_{2, j}(x))
+g_1(0)\lambda_1  \int^{L}_{0} \frac{ p_{0}(t -\tilde{p}^*_{1, j}(x)-\tilde{p}^*_{2, j}(x)) }{g_1(x)}\,dx\right).
\end{align*}
Note that the integral term in $p_{2,j}(x, t)$  only depends on $t$.
Therefore, 
\begin{align}
j \alpha_i \frac{\partial (g_i(x)p_{i,j}(x,t))}{\partial x}&=j\alpha_i g_i(0)\lambda_i\frac{dp_0(  t-\tilde{p}^*_{i,j}(x))}{dt}\Big(-\frac{p^*_{i}(x)}{j\alpha_i }\Big)\nonumber\\
&= -g_i(0)\lambda_i\frac{dp_0(  t-\tilde{p}^*_{i,j}(x))}{dt}p^*_{i}(x).\label{EST_p_bd}
\end{align}
%and
%\begin{align}
%\frac{\partial (j \alpha_2g_2(x)p_{2,j}(x,t))}{\partial x}
%=&j \alpha_2 \lambda_{2} g_2(0)\frac{d p_{0}(t-\tilde{p}^*_{2, j}(x))}{dt}\Big(-\frac{p^*_2(x)}{j\alpha_2}\Big)\nonumber\\
%&=\lambda_{2} g_2(0)\frac{d p_{0}(t-\tilde{p}^*_{2, j}(x))}{dt}\Big(-p^*_2(x)\Big).\label{EST_p2_bd}
%\end{align}
Since $p^*_{i}(x), i=1,2,$ are bounded for $x \in [0, l]$, it remains to show that  
\begin{align}
\sup_{x\in [0, l]} \Big|\frac{dp_{0, j}(  t-\tilde{p}^*_{i,j}(x))}{dt} \Big|<\infty, \quad i=1,2, \quad \text{as} \quad j\to \infty.
\label{EST_diff_p0_bd}
\end{align}

With the help of \eqref{closed_loop_p0},  we get
\begin{align*}
\sup_{x\in [0, l]}\Big|\frac{dp_{0, j}(  t-\tilde{p}^*_{i,j}(x))}{dt}\Big| \leq 
&j\sum^2_{n=1} \alpha_n \sup_{x\in [0, l]} \Big| g_n(L)p_{n, j}(L, t-\tilde{p}^*_{i,j}(x))- g_n(0)p_{n, j}(0, t-\tilde{p}^*_{i,j}(x) )\Big|,
\end{align*}
for $i=1,2$.
Combining  the exponential decay result in \eqref{EST_key} together with \eqref{closed_loop_p1}--\eqref{closed_loop_p2} for
 $t_f > 2\sum^2_{i=1} \tilde{p}^*_{i}(L)$, we have
\begin{align}
&\sup_{x\in [0, l]} \Big| g_1(L)p_{1}(L, t-\tilde{p}^*_{i,j}(x))- g_1(0)p_{1}(0, t-\tilde{p}^*_{i,j}(x) )\Big|\nonumber\\
&\quad =\sup_{x\in [0, l]}  \Big| g_1(0)\lambda_1  p_{0, j}(t-\tilde{p}^*_{i,j}(x) -\tilde{p}^*_{i, j}(L)) - g_1(0)\lambda_1 p_{0}( t-\tilde{p}^*_{i,j}(x))\Big|\nonumber\\
&\quad \leq g_1(0)\lambda_1  \sup_{x\in [0, l]}\Big(\Big|p_{0,j}( t-\tilde{p}^*_{i,j}(x)- \tilde{p}^*_{i,j}(L) )-p^*_0\Big| 
 +\Big|p_{0,j}(  t-\tilde{p}^*_{i,j}(x)) -p^*_0\Big|\Big) \nonumber\\
&\quad \leq g_1(0)\lambda_1   \Big(  \sup_{\tau\in [t-\tilde{p}^*_{i,j}(l)- \tilde{p}^*_{i,j}(L) , t- \tilde{p}^*_{i,j}(L) ]}\big|p_{0,j}( \tau )-p^*_0\big| 
+\sup_{\tau\in [t-\tilde{p}^*_{i,j}(l), t]}\big|p_{0,j}( \tau) -p^*_0\big|\Big) \nonumber\\
&\quad \leq 2g_1(0)\lambda_1   \sup_{\tau\in [t-2 \tilde{p}^*_{i,j}(L) , t]}\big|p_{0,j}( \tau )-p^*_0\big| \nonumber\\
&\quad \leq 2g_1(0)\lambda_1   \|\vec{p}(\cdot, t-2 \tilde{p}^*_{i,j}(L) )-\vec{p}^*\|_{X}.\label{EST_diff_gp1_CL}
\end{align}
Following the same approach as in \eqref{EST_diff_gp1_CL}, we have
\begin{align}
&\sup_{x\in [0, l]} \Big| g_2(L)p_{2}(L, t-\tilde{p}^*_{i,j}(x))- g_2(0)p_{2}(0, t-\tilde{p}^*_{i,j}(x) )\Big|\nonumber\\
%====details====
%&\quad=\sup_{x\in [0, l]}  \Big| \lambda_{2}  g_2(0)\Big(p_{0, j}(t-\tilde{p}^*_{i,j}(x)-\tilde{p}^*_{i, j}(L))\\
%&\qquad+g_1(0)\lambda_1 \int^{L}_{0} \frac{  p_{0, j}(t-\tilde{p}^*_{i,j}(x) -\tilde{p}^*_{1,j}(x) -\tilde{p}^*_{2, j}(x)) }{g_1(x)}\,dx\Big)\nonumber\\
%&\qquad- g_2(0)\lambda_{2}\Big (p_{0,j}(t-\tilde{p}^*_{2,j}(x))
%+g_1(0)\lambda_1 \int^{L}_{0}\frac{p_{0, j}(t- \tilde{p}^*_{i,j}(x) -\tilde{p}^*_{1, j}(x)-\tilde{p}^*_{2, j}(x))}{ g_1(x)}  \, dx\Big)\Big|
%\nonumber\\
%&\quad\leq g_2(0)\lambda_2 \Big[ \sup_{x\in [0, l]}\Big(\Big|p_{0,j}( t-\tilde{p}^*_{i,j}(x)- \tilde{p}^*_{i,j}(L) )-p^*_0\Big| 
 %+\Big|p_{0,j}(  t-\tilde{p}^*_{i,j}(x)) -p^*_0\Big|\Big) \Big]\nonumber\\
 %====details====
 &\quad\leq 2   g_2(0)\lambda_2  \|\vec{p}(\cdot, t-2 \tilde{p}^*_{i,j}(L) )-\vec{p}^*\|_{X}. \label{EST_diff_gp2_CL}
\end{align}
Thus from \eqref{EST_diff_gp1_CL}--\eqref{EST_diff_gp2_CL} it follows
\begin{align}
\sup_{x\in [0, l]}\Big|\frac{dp_{0, j}(  t-\tilde{p}^*_{i,j}(x))}{dt}\Big| 
\leq 2j\big(\sum^2_{n=1} \alpha_n g_n(0)\lambda_n \big) \|\vec{p}(\cdot, t-2 \tilde{p}^*_{i,j}(L) )-\vec{p}^*\|_{X}, \quad i=1,2. \label{EST_diff_p0j}
\end{align}
Recall that  $ \int^L_0p^*_i(x)\,dx\leq 1$ and  $\alpha_i\geq \frac{1}{r_0}, i=1,2,$ by \eqref{new_alpha}. We have 
 $$\tilde{p}^*_{i,j}(L)=\frac{1}{j\alpha_i}\int^L_0p^*_i(x)\,dx\leq \frac{r_0}{j}, \quad i=1,2.$$
Let $j\geq 2$. Then
 $ t-2\tilde{p}^*_{i,j}(L)\geq t_{j-1}-\frac{2r_0}{j}$.
In light of Corollary \ref{cor2} and \eqref{EST_key} we get  
\begin{align*}
 & \|\vec{p}(\cdot, t-2 \tilde{p}^*_{i,j}(L) )-\vec{p}^*\|_{X}% \leq \|\vec{p}(x, t_{j-1}-\frac{2r_0}{j})-\vec{p}^*\|_{X} \nonumber\\
 \leq M_ce^{-\int^{t-2 \tilde{p}^*_{i,j}(L)}_0 \varepsilon_j(\tau) \,d\tau}\nonumber\\
 &\qquad\leq M_ce^{-\int^{t_{j-1}-\frac{2r_0}{j}}_0 \varepsilon_j(\tau) \,d\tau}\nonumber\\
 &\qquad=M_c( e^{- \int^{t_{j-1}}_0   \varepsilon_j(\tau) \,d\tau}\cdot e^{ \int^{t_{j-1}}_{t_{j-1}-\frac{2r_0}{j}}  \varepsilon_j(\tau)\,d\tau}) \nonumber\\
&\qquad\leq M_ce^{-  \varepsilon_{c} r_0\sum^{j-1}_{k=1} \frac{1}{k} }\cdot e^{ (j-1)\varepsilon_c \frac{2r_0}{j}},
\end{align*}
where $ \varepsilon_j(\tau)=j\varepsilon_c$ for $\tau\in [t_{j-1}, t_j), j\in \mathbb{N}^+$.
Consequently, 
\begin{align}
&\sup_{x\in [0, l]}\Big|\frac{dp_{0, j}(  t-\tilde{p}^*_{i,j}(x))}{dt}\Big| 
\leq 2j\big(\sum^2_{n=1} \alpha_n g_n(0)\lambda_n \big)  M_ce^{-\varepsilon_{c} r_0\sum^{j-1}_{k=1} \frac{1}{k} }\cdot e^{ (j-1)\varepsilon_c \frac{2r_0}{j}},
 \label{EST_bd}
\end{align}
where 
\begin{align}
\lim_{j\to \infty}e^{ (j-1)\varepsilon_c \frac{2r_0}{j}}=e^{2r_0 \varepsilon_c }.\label{EST_bd2}
\end{align}
 It remains to analyze the property of 
$j e^{- \varepsilon_{c} r_0\sum^{j-1}_{k=1}   \frac{1}{k} }$ when $j$ is sufficiently large. 
Let $J=j-1$. Then 
\begin{align*}
&j e^{- \varepsilon_{c} r_0\sum^{j-1}_{k=1}   \frac{1}{k} } =(J+1)  e^{-  \varepsilon_{c} r_0\sum^{J}_{k=1}   \frac{1}{k} } \\
&= J e^{-  \varepsilon_{0} r_0\sum^{J}_{k=1}   \frac{1}{k} }+e^{- \varepsilon_{c} r_0\sum^{J}_{k=1}   \frac{1}{k} }\\
&=e^{-(-\ln J+ \varepsilon_{c} r_0 \sum^{J}_{k=1}   \frac{1}{k} )}+ e^{-  \varepsilon_{c} r_0\sum^{J}_{k=1}   \frac{1}{k} }.
%&\leq e^{-(-\ln j+ \sum^{j}_{k=1}   \frac{1}{k} )}+ e^{-\sum^{j}_{k=1}   \frac{1}{k} }.
\end{align*}
Since 
$ \lim_{J\to\infty} (-\ln J+ \sum^{J}_{k=1}   \frac{1}{k} )=\gamma>0$
  is the Euler-Mascheroni constant \cite[Sec.1.5]{finch2003mathematical}, 
 if  $\varepsilon_c\geq\frac{1}{ r_0}$,  then 
 \begin{align}
 &\lim_{J\to \infty}(e^{-(-\ln J+  \varepsilon_{0} r_0 \sum^{J}_{k=1}   \frac{1}{k} )}+ e^{-  \varepsilon_{c} r_0\sum^{J}_{k=1}   \frac{1}{k} }) \leq e^{-\gamma}. \label{EST_bd3}
 \end{align}
 The condition $\varepsilon_c\geq\frac{1}{ r_0}$ can be always achieved by increasing both $\alpha_1$ and $\alpha_2$.  Finally, combining  \eqref{EST_bd} with  \eqref{EST_bd2}--\eqref{EST_bd3} yields the claim  \eqref{EST_diff_p0_bd}.
It is worth to point that  if $\varepsilon_c r_0=1+\eta$ for some $\eta>0$, then 
\eqref{EST_bd3} converges to zero as $J\to \infty$, and hence \eqref{EST_diff_p0_bd} converges to zero. Consequently,
\[\mu_i(x,t)\to -\frac{p^*_{i_x}(x)}{p^*_{i}(x)}\quad \text{as}\quad j\to \infty, \ \text{i.e.,}\ t\to t_f, \quad i=1,2,\]
and   this completes the proof. 
\end{proof}
As a result of Proposition \ref{prop4}, Theorem \ref{thm_main2} holds immediately and this concludes our current work. 

%===============================================================================
\section{Conclusion }\label{conclusion}
In this work, we have established  the well-posedness of a repairable system with a degraded state and its bilinear controllability   via system repair rates.  The repair rates are constructed in feedback forms. Our approach essentially makes use of  the exponential convergence of the closed-loop system solution to its steady-state and then weights the repair actions in time as to steer the system to the desired distribution in a finite time interval.  It is worth to point out that  there are many other ways of choosing the control weights in time in  \eqref{feedback_mu1}--\eqref{feedback_mu2}, as long as the series in \eqref{EST_key} diverges. Our analysis  mainly employs the classic method of characteristics and the $C_0$-semigroup tools. The control design is  generic and applicable to a general family of repairable systems that share the similar attributes. %===============================================================================

\section{Appendix}\label{app}
%\begin{proof} 
{\it Proof of Proposition \ref{prop1} (1):  $ia\in \rho(\mathcal{A})$ for $a\neq 0$}.

Suppose that there exists $r=ia$ with $a\neq 0$, such that $\phi(r)=0$. Then  by \eqref{EST_Phi_r},
\begin{align*}
1&+\lambda_1 \int_0^L(\cos(ax)-i\sin(ax))e^{-\int_0^x(\lambda_{2}+\mu_1(s))\,ds}dx\\
&+\lambda_{2}\int_0^{L}(\cos(ax)-i\sin(ax))e^{-\int_0^x\mu_2(s)\,ds}dx\\
&+\lambda_{1}\lambda_{2}\left(\int_0^{L}(\cos(ax)-i\sin(ax)) e^{-\int_0^x\mu_2(s)\,ds}\,dx\right)\\
&\qquad\qquad\cdot\left(\int^{L}_{0}(\cos(ax)
-i\sin(ax)) e^{-\int_0^x(\lambda_{2}+\mu_1(s))\,ds}\,d x\right)=0,
\end{align*}
where the real components satisfy 
\begin{align*}%\label{eq: real part}
1&+\lambda_1 \int_0^L\cos(ax)e^{-\int_0^x(\lambda_{2}+\mu_1(s))ds}dx+\lambda_{2}\int_0^{L}\cos(ax)e^{-\int_0^x\mu_2(s)\,ds}\,dx\nonumber\\
&+\lambda_1\lambda_{2}\left(\int_0^{L}\cos(ax) e^{-\int_0^x\mu_2(s)ds}dx\right)\left(\int^{L}_{0}\cos(ax) e^{-\int_0^x(\lambda_{2}+\mu_1(s))\,ds}\,d x\right)\nonumber\\
&-\lambda_1\lambda_{2}\left(\int_0^{L}\sin(ax) e^{-\int_0^x\mu_2(s)ds}dx\right)\left(\int^{L}_{0}\sin(ax) e^{-\int_0^x(\lambda_{2}+\mu_1(s))\,ds}\,d x\right)=0,
%\label{eq_re}
\end{align*}
or
\begin{align}%\label{eq: real part}
&\left(1+\lambda_{2}\int_0^{L}\cos(ax) e^{-\int_0^x\mu_2(s)ds}\,dx\right)\left(1+\lambda_1 \int^{L}_{0}\cos(ax) e^{-\int_0^x(\lambda_{2}+\mu_1(s))\,ds}\,d x\right)\nonumber\\
&\quad-\lambda_1\lambda_{2}\left(\int_0^{L}\sin(ax) e^{-\int_0^x\mu_2(s)ds}\,dx\right)\left(\int^{L}_{0}\sin(ax) e^{-\int_0^x(\lambda_{2}+\mu_1(s))\,ds}\,d x\right)=0,
\label{2eq_re}
\end{align}
and the imaginary components satisfy 
\begin{align}
&\lambda_1 \int_0^L\sin(ax)e^{-\int_0^x(\lambda_{2}+\mu_1(s))\,ds}\,dx
+\lambda_{2}\int_0^{L}\sin(ax)e^{-\int_0^x\mu_2(s)ds}\,dx\nonumber\\
&\quad+\lambda_{1}\lambda_{2}\int_0^{L}\sin(ax) e^{-\int_0^x\mu_2(s)ds}\,dx \left(\int^L_0\cos(ax)e^{-\int_0^x(\lambda_{2}+\mu_1(s))\,ds}\,dx\right) \nonumber\\
&\quad+\lambda_{1}\lambda_{2}\int_0^{L}\cos(ax) e^{-\int_0^x\mu_2(s)ds}dx\left(\int^{L}_{0}\sin(ax) e^{-\int_0^x(\lambda_{2}+\mu_1(s))\,ds}\,dx\right)
=0. \label{eq_im}
\end{align}
From \eqref{eq_im} we further have
\begin{align}
&\lambda_1\left(1+\lambda_{2}\int_0^{L}\cos(ax) e^{-\int_0^x\mu_2(s)ds}\,dx\right) \left(\int_0^L\sin(ax)e^{-\int_0^x(\lambda_{2}+\mu_1(s))ds}\,dx\right)\nonumber\\
&\quad+\lambda_{2}\left(\int_0^{L}\sin(ax) e^{-\int_0^x\mu_2(s)ds}\,dx\right) 
\left( 1+\lambda_{1} \int^L_0\cos(ax)e^{-\int_0^x(\lambda_{2}+\mu_1(s))ds}\,dx\right) 
=0. \label{2eq_im}
\end{align}
To simply the formulations, we let 
\begin{align*}
&I_1=1+\lambda_{2}\int_0^{L}\cos(ax) e^{-\int_0^x\mu_2(s)ds}\,dx,  \quad
I_2=1+\lambda_1 \int^{L}_{0}\cos(ax) e^{-\int_0^x(\lambda_{2}+\mu_1(s))\,ds}\,d x\\
& I_3=\int_0^{L}\sin(ax) e^{-\int_0^x\mu_2(s)\,ds}\,dx,\quad  \text{and}\quad 
I_4=\int^{L}_{0}\sin(ax) e^{-\int_0^x(\lambda_{2}+\mu_1(s))\,ds}\,d x.
\end{align*}
Using \eqref{2EST_mu} and integration by parts we have
\begin{align}
&I_3%=\frac{1}{a}-\frac{1}{a}\int_0^L\cos(ax) \mu_2(x)e^{-\int_0^x\mu_2(\alpha)\,ds}\,dx\\
=\frac{1}{a} \int_0^L(1-\cos(ax)) \mu_2(x)e^{-\int_0^x\mu_2(\alpha)\,ds}\,dx,\label{1EST_sin}\\
 &I_4
 % =-\int_0^Le^{-\int_0^x(ds\mu_1(s)+\lambda_{2})ds}\,d(\frac{\cos(ax)}{a})\\
%  &\qquad=\frac{1}{a}-\frac{1}{a}\int_0^L\cos(ax) (\mu_1(x)+\lambda_2)e^{-\int_0^x(ds\mu_1(s)+\lambda_{2})ds}\,dx\\
  =\frac{1}{a}\int_0^L(1-\cos(ax)) (\lambda_2+\mu_1(x)) e^{-\int_0^x(\lambda_{2} +\mu_1(s)ds}\,dx.\label{2EST_sin}
\end{align}

Combining \eqref{2eq_re} with \eqref{1EST_sin}--\eqref{2EST_sin} follows  that
\begin{align}
I_1I_2=&\lambda_1\lambda_{2}I_3I_4>0, \label{3eq_re}
\end{align}
for $a\neq0$. However, from \eqref{2eq_im} we have
\begin{align*}
\frac{I_1}{I_2}=-\frac{\lambda_{2}}{\lambda_1}\frac{I_3}{I_4}<0, %\label{3eq_im}
\end{align*}
which contradicts  with \eqref{3eq_re}. Therefore, $\phi(r)\neq0$ for $r=ia$ where $a\neq 0$, and thus $r\in \rho(\mathcal{A})$. This completes the proof.
%\end{proof}
 \hfill $\Box$

%\newlins e 

\end{document}